\documentclass[a4paper]{amsart}
\usepackage{txfonts, amsmath,amstext,amsthm,amscd,amsopn,verbatim,amssymb, amsfonts}
\usepackage{fullpage}

\usepackage{float}
\restylefloat{table}

\usepackage[bbgreekl]{mathbbol}
\usepackage{tikz}
\usetikzlibrary{matrix}
\usetikzlibrary{shapes}
\usetikzlibrary{arrows}
\usetikzlibrary{calc,3d}
\usetikzlibrary{decorations,decorations.pathmorphing,decorations.pathreplacing,decorations.markings}
\usetikzlibrary{through}
\tikzset{ext/.style={circle, draw,inner sep=1pt},int/.style={circle,draw,fill,inner sep=1.4pt},nil/.style={inner sep=1pt}}
\tikzset{cy/.style={circle,draw,fill,inner sep=2pt},scy/.style={circle,draw,inner sep=2pt},scyx/.style={draw,cross out,inner sep=2pt},scyt/.style={draw,regular polygon,regular polygon sides=3,inner sep=0.95pt}}
\tikzset{exte/.style={circle, draw,inner sep=3pt},inte/.style={circle,draw,fill,inner sep=3pt}}
\tikzset{diagram/.style={matrix of math nodes, row sep=3em, column sep=2.5em, text height=1.5ex, text depth=0.25ex}}
\tikzset{diagram2/.style={matrix of math nodes, row sep=0.5em, column sep=0.5em, text height=1.5ex, text depth=0.25ex}}

\newcommand{\nomul}{{
\begin{tikzpicture}[baseline=-.55ex,scale=.2]
 \node[circle,draw,fill,inner sep=.5pt] (a) at (0,0) {};
 \node[circle,draw,fill,inner sep=.5pt] (b) at (1,0) {};
 \draw (a) edge[bend left=25] (b);
 \draw (a) edge[bend right=25] (b);
 \draw (.2,-.5) -- (.8,.6);
\end{tikzpicture}}}

\newcommand{\notadp}{{
\begin{tikzpicture}[baseline=-.55ex,scale=.2, every loop/.style={}]
 \node[circle,draw,fill,inner sep=.5pt] (a) at (0,0) {};
 \draw (a) edge[loop] (a);
 \draw (-.2,-.2) -- (.2,.5);
\end{tikzpicture}}}

\newcommand{\mxto}[1]{{\overset{#1}{\longmapsto}}}

\theoremstyle{plain}
  \newtheorem{thm}{Theorem}
  
  \newtheorem{prop}{Proposition}

  \newtheorem{lemma}{Lemma}
\theoremstyle{definition}
  
  \newtheorem*{rem}{Remark}

\newcommand{\alg}[1]{\mathfrak{{#1}}}

\newcommand{\K}{{\mathbb{K}}}
\newcommand{\Z}{{\mathbb{Z}}}
\newcommand{\N}{{\mathbb{N}}}

\newcommand{\fGC}{\mathrm{fGC}}

\newcommand{\bpm}{\begin{pmatrix}}
\newcommand{\epm}{\end{pmatrix}}

\newcommand{\GC}{\mathrm{GC}}

\newcommand{\GCn}{\mathrm{GCn}}

\newcommand{\QP}[1]{\left( #1 \right)_\infty}

\newcommand{\VV}{V}

\DeclareMathOperator{\sgn}{sgn}

\newcommand{\gra}{\mathit{gra}}
\newcommand{\grt}{\alg {grt}}

\DeclareMathOperator{\LCM}{lcm}
\DeclareMathOperator{\GCD}{gcd}

\begin{document}
\title{The dimensions and Euler characteristics of M. Kontsevich's graph complexes}
\author{Thomas Willwacher}
\address{Institute of Mathematics\\ University of Zurich\\ Winterthurerstrasse 190 \\ 8057 Zurich, Switzerland}
\email{thomas.willwacher@math.uzh.ch}

\author{Marko \v Zivkovi\' c}
\address{Institute of Mathematics\\ University of Zurich\\ Winterthurerstrasse 190 \\ 8057 Zurich, Switzerland}
\email{marko.zivkovic@math.uzh.ch}


\begin{abstract}
We provide a generating function for the (graded) dimensions of M. Kontsevich's graph complexes of ordinary graphs.
This generating function can be used to compute the Euler characteristic in each loop order.
Furthermore, we show that graphs with multiple edges can be omitted from these graph complexes.
\end{abstract}

\maketitle

\section{Introduction}
The graph complexes in its various guises are some of the most mysterious and fascinating objects in mathematics. 
They are combinatorially very simple to define, as linear combinations of graphs of a certain kind, with the operation of edge contraction or dually, vertex splitting as differential.
Still, their cohomology is very hard to compute and largely unknown at present.

There are various versions of graph cohomology, each playing a central role in one or more fields of algebra, topology or mathematical physics.
\begin{itemize}
 \item Ordinary (non-decorated) graph cohomology describes the deformation theory of the $E_n$ operads and plays a central role in many quantization problems \cite{grt, Kformal, Kgrcomplex}.
 \item Ribbon graph complexes describe the cohomology of the moduli spaces of curves \cite{Kformal}.
 \item Lie decorated graph complexes describe the cohomology of the automorphisms of a free group and play a central role in many results in low dimensional topology \cite{cullervogtmann,Kformal}.
 \item Graph complexes of the above sorts, but with external legs, also have a wide range of applications in topology, knot theory and algebra \cite{LambrechtsTurchin,Turchin1, Turchin2, Turchin3, BerglundMadsen,vogtmannhairy}.
\end{itemize}
In all cases, the differential leaves the genus (i.~e., the first Betti number, or loop order) of the graph invariant. Hence the graph complexes split into a direct sum of subcomplexes according to genus, and all the above graph cohomology spaces are (at least) bigraded, by the cohomological degree and the genus.

In none of the above cases is the graph cohomology known.
 Our current knowledge essentially comes from the following sources:
\begin{itemize}
 \item In low degrees, one knows the cohomology by computer experiments. \cite{barnatanmk}
 \item By degree considerations one can show that the graph cohomology is concentrated in a certain range of allowed bidegrees.
 \item For several, but not all flavors of graph cohomology there are formulas for the Euler characteristics. \cite{smillievogtmann, LambrechtsTurchin, harerzagier}.
 \item There are many families of known graph cocycles. Some of these cocycles are known to represent non-trivial cohomology classes, while for others this is a conjecture.
 \item On some graph complexes there is known to be additional algebraic structure. For example the complex $\GC_n$ of ordinary graphs (see below) is naturally a differential graded Lie algebra. Hence the algebraic operations may be used to construct new cohomology classes out of known classes.
 \item The first author showed in \cite{grt} that $H^0(\GC_2)\cong \grt_1$. This is the only case we know in which a classical graph cohomology is known in one degree where it is non-trivial.
\end{itemize}

In this paper we focus on the ``ordinary'' graph complexes introduced by M. Kontsevich.
The elements of these complexes are linear combinations of isomorphism classes of undirected graphs, with at least trivalent vertices, and without ``odd'' symmetries. There are two natural choices of when to count a symmetry of a graph as odd, thus yielding two distinct complexes. First, a symmetry may be counted as odd if its induced permutation on the set of edges is odd.
Secondly, a symmetry may be considered odd if its induced permutation on the set of vertices and half-edges is odd.
For example, the left hand graph in the following picture has an odd symmetry according to the first convention, but not the second, and vice versa for the right hand graph.
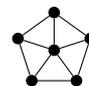
\begin{align*}
 \Theta 
 &:=
 \begin{tikzpicture}[baseline=-.65ex, scale=.7]
  \node[int] (v) at (0,0){};
  \node[int] (w) at (1,0){};
  \draw (v) edge (w) edge[bend left=35] (w) edge[bend right=35] (w);
 \end{tikzpicture}
 &
 &
 \begin{tikzpicture}[baseline=-.65ex, rotate=90, scale=.5]
  \node[int] (v) at (0,0){};
  \node[int] (w1) at (0:1){};
  \node[int] (w2) at (72:1){};
  \node[int] (w3) at (144:1){};
  \node[int] (w4) at (216:1){};
  \node[int] (w5) at (288:1){};
  \draw (v) edge (w1) edge (w2) edge (w3) edge (w4) edge (w5)
        (w1) edge (w2) edge (w5) 
        (w3) edge (w4) edge (w2) 
        (w4) edge (w5);
 \end{tikzpicture}
\end{align*}

We will denote by $\VV_{v,e}^{even}$, respectively $\VV_{v,e}^{odd}$ the vector spaces spanned by isomorphism classes of at least trivalent graphs without ``odd'' symmetries for the first, respectively the second notion of oddness, with $v$ vertices and $e$ edges. We allow multiple edges or tadpoles (short cycles).
The first main result of this paper are generating functions for the dimensions of $\VV_{v,e}^{even}$, respectively $\VV_{v,e}^{odd}$, i.~e., for the numbers of isomorphism classes of graphs as above. 

\begin{thm}\label{thm:dimgenfun}
Define
\begin{align*}
 P^{odd}(s,t) &:= \sum_{v,e} \dim\left(\VV_{v,e}^{odd}\right) s^v t^e &
 P^{even}(s,t) &:= \sum_{v,e} \dim\left(\VV_{v,e}^{even}\right) s^v t^e \, .
\end{align*}
Then
\begin{align*}
 P^{odd}(s,t) &:= 
\frac{1}{\QP{-s, (st)^2} \QP{(st)^2,(st)^2}}
\sum_{j_1,j_2,\dots \geq 0}
\prod_\alpha \frac{(-s)^{\alpha j_\alpha}}{j_\alpha! (-\alpha)^{j_\alpha}  }
\frac{1}{\QP{(-st)^\alpha, (-st)^\alpha}^{j_\alpha}}
\left( \frac{\QP{t^{2\alpha-1}, (st)^{4\alpha-2}}}{\QP{(-s)^{2\alpha-1}t^{4\alpha-2},(st)^{4\alpha-2}}} \right)^{j_{2\alpha-1}/2}
\\&\quad\quad\quad
\left( \frac{\QP{t^{\alpha}, (st)^{2\alpha}}}{ \QP{(-s)^{\alpha}t^{2\alpha},(st)^{2\alpha}}} \right)^{j_{2\alpha}}
\prod_{\alpha, \beta} \frac{1}{\QP{t^{\LCM(\alpha, \beta)}, (-st)^{\LCM(\alpha, \beta)}}^{\GCD(\alpha, \beta)j_\alpha j_\beta/2}},
\\
 P^{even}(s,t) &:= 
\frac{\QP{s, (st)^2}}{\QP{-st,(st)^2}}
\sum_{j_1,j_2,\dots \geq 0}
\prod_\alpha \frac{s^{\alpha j_\alpha}}{j_\alpha! \alpha^{j_\alpha}  }
\frac{1}{\QP{(-st)^\alpha, (-st)^\alpha}^{j_\alpha}}
\left( \frac{\QP{(-t)^{2\alpha-1}, (st)^{4\alpha-2}}}{\QP{s^{2\alpha-1}t^{4\alpha-2},(st)^{4\alpha-2}}} \right)^{j_{2\alpha-1}/2}
\\&\quad\quad\quad
\left( \frac{\QP{(-t)^{\alpha}, (st)^{2\alpha}}}{ \QP{s^{\alpha}t^{2\alpha},(st)^{2\alpha}}} \right)^{j_{2\alpha}}
\prod_{\alpha, \beta} \QP{(-t)^{\LCM(\alpha, \beta)}, (-st)^{\LCM(\alpha, \beta)}}^{\GCD(\alpha, \beta)j_\alpha j_\beta/2}
\end{align*}
where 
\begin{align*}
\QP{a,q} &= \prod_{k\geq 0}\left(1-a q^k\right)
\end{align*}
is the $q$-Pochhammer symbol.
\end{thm}

These formulas may be used to compute the Euler characteristics of the graph complexes, which can be found in Table \ref{tbl:eulerchar} below for loop orders up to 30.
These results in particular allow us to probe the graph cohomology far beyond the region where it is currently accessible to direct computer calculation.

The second main result of this paper is to show that the graph complexes $\GC_{2n+1}$ (see below for the definition) can be significantly simplified essentially without altering their cohomology by omitting all isomorphism classes of graphs with multiple edges. Let $\GC_{2n+1}^\nomul\subset \GC_{2n+1}$ the subcomplex spanned by graphs without multiple edges.

\begin{thm}\label{thm:intromultiple}
 The inclusion $\K\Theta\oplus \GC_{2n+1}^\nomul \to \GC_{2n+1}$ is a quasi-isomorphism of complexes.
\end{thm}

\begin{rem}
 S. Morita, T. Sakasai and M. Suzuki recently computed the Euler characteristics of the commutative graph complexes up to weight 16, using different methods \cite{MSS}. 
\end{rem}

\subsection{Structure of the paper}
In section \ref{sec:grcomplexes} we recall the definitions of the graph complexes we study in this paper.
Section \ref{sec:multiple} is dedicated to the proof of Theorem \ref{thm:intromultiple}, while in section \ref{sec:eulerchar} we give a proof of Theorem \ref{thm:dimgenfun}.


\subsection*{Acknowledgements}
We thank P. Etingof, who suggested the problem of computing the Euler characteristics of the graph complexes to T.W. some years ago. We also thank V. Turchin for valuable discussions.
This work was partially supported by the Swiss National Science Foundation, grants PDAMP2\_137151 and 200021\_150012.

\section{Graph complexes}\label{sec:grcomplexes}

\subsection{Graph vector spaces}
For $v,e\in\N$ let $\gra_{v,e,0}$ be the set of all graphs with $v$ distinguishable vertices and $e$ distinguishable oriented edges between vertices, i.~e. $\gra_{v,e,0}$ is the set of all maps $g\colon K\to E\times E$ where $V=\{1,2,\dots v\}$ is the set of vertices and $E=\{1,2,\dots,e\}$ is the set of edges. We say that edge $e$ connects vertex $g_1(e)$ to the vertex $g_2(e)$. Let $V_{v,e,0}$ be vector spaces freely generated by graphs in $\gra_{v,e,0}$.

For $v,e\in\N$ and $i\in\N$ let $\gra_{v,e,i}$ be the set of all graphs in $\gra_{v,e,0}$ such that every vertex is adjacent to at least $i$ edges, i.~e.\ for every vertex $v\in N$ there are at least $i$ pairs $(e,j)\in K\times\{1,2\}$ such that $g_j(e)=v$. We call the number of adjacent edges the \emph{valence} of a vertex. Let $V_{v,e,i}$ be vector space freely generated by graphs in $\gra_{v,e,i}$.

However, we are not interested in distinguishing vertices, or edges, or directions of edges. First, if we reverse the orientation of an edge of a graph $\Gamma$ and obtain the graph $\Gamma'$, we identify either $\Gamma=\Gamma'$ or $\Gamma=-\Gamma'$. To do this we introduce the representations $\nu^+$ of the group $(S_2)^e$ on the space $V_{v,e,i}$ by reversing the orientation of edges. Let $\nu^-(g)=(-1)^t\nu^+(g)$, where $t$ is the number of edges reversed by $g\in(S_2)^e$, be another representation of the group $S_2^e$ on the space $V_{v,e,i}$. Let $V_{v,e,i}^{s_\nu}=$ for $s_\nu\in\{+,-\}$ be the space of coinvariants of the representation $\nu^{s_\nu}$.

Secondly, let $\mu^+$ be the representations of the symmetric group $S_e$ on the space $V_{v,e,i}$ which interchange edges. Let $\mu^-(g)=\sgn(g)\mu^+(g)$, where $\sgn(g)$ is the sign of the permutation $g\in S_e$, be another representations of the group $S_e$ on the space $V_{v,e,i}$. 
It is easy to see that representations $\mu^{s_\mu}$ for $s_\mu\in\{+,-\}$ descend to representations on the spaces of coinvariants $V_{v,e,i}^{s_\nu}$.
Let $V_{v,e,i}^{s_\nu s_\mu}$ be the space of coinvariants of the representation $\mu^{s_\mu}$ on $V_{v,e,i}^{s_\nu}$.

Finally, let $\rho^+$ be the representations of the symmetric group $S_v$ on the space $V_{v,e,i}$ by interchanging vertices. Let $\rho^-(g)=\sgn(g)\rho^+(g)$ be another representation of the group $S_v$ on the space $V_{v,e,i}$. It is easy to see that representations $\rho^{s_\rho}$ for $s_\rho\in\{+,-\}$ descend to the spaces of coinvariants $V_{v,e,i}^{s_\nu s_\mu}$. We are interested in the space $V_{v,e,i}^{s_\nu s_\mu s_\rho}$, the space of coinvariants of the representation $\rho^{s_\rho}$ on $V_{v,e,i}^{s_\nu s_\mu}$.

A basis of $V_{v,e,i}^{+++}$ consists of all graphs with $v$ indistinguishable vertices and $e$ unoriented indistinguishable edges (i.~e., isomorphism classes of graphs), such that every vertex is adjacent to at least $i$ edges. A basis of for instance $V_{v,e,i}^{+-+}$ consists of the subset of (isomorphism classes of) graphs that do not have automorphisms which act by an odd permutation on the edges.

Note that if $s_\mu=-$, then every graph which contains multiple edge has an odd automorphism, by interchanging the constituent edges of the multiple edge. Hence such graphs are zero in the complexes $V_{v,e,i}^{s_\nu - s_\rho}$. Similarly, if $s_\nu=-$, then reverting the direction of a tadpole, i.~e, an edge $e$ such that $g_1(e)=g_2(e)$, is an odd automorphism and hence such graphs are zero in the complexes $V_{v,e,i}^{- s_\mu s_\rho}$.
We may be interested in excluding tadpoles or multiple edges even in the cases $s_\nu=+$ and $s_\mu=+$ by starting with the space of graphs which exclude them instead of $\gra_{v,e}$. Let $V_{v,e,i}^{* s_\mu s_\rho}\subset V_{v,e,i}^{+ s_\mu s_\rho}$ and $V_{v,e,i}^{* s_\mu}\subset V_{v,e,i}^{+ s_\mu}$ be the subspaces spanned by graphs without tadpoles, and let  $V_{v,e,i}^{s_\nu * s_\rho}\subset V_{v,e,i}^{s_\nu + s_\rho}$ and $V_{v,e,i}^{s_\nu *}\subset V_{v,e,i}^{s_\nu +}$ be the subspaces spanned by graphs without multiple edges. Thus, from now on we consider $s_\nu,s_\mu\in\{+,*,-\}$.

\subsection{Chain complexes}

Let $d\colon V_{v,e,0}^{s_\nu s_\mu s_\rho}\rightarrow V_{v+1,e+1,0}^{s_\nu s_\mu s_\rho}$ for $s_\nu,s_\mu\in\{+,-\}$ be the map
\begin{equation}
 d\Gamma\,=\sum_{x\in V(\Gamma)}\frac{1}{2}(\text{splitting of $x$}) - (\text{adding an edge at $x$}),
\end{equation}
where ''splitting of $x$" means putting
\begin{tikzpicture}[scale=.5]
 \node[int] (a) at (0,0) {};
 \node[int] (b) at (1,0) {};
 \draw (a) edge[->] (b);
 \node[above left] at (a) {$\scriptstyle x$};
 \node[above right] at (b) {$\scriptstyle v+1$};
\end{tikzpicture}
instead of the vertex $x$ and summing over all possible ways how to connect edges that have been connected to $x$ to the new $x$ and $v+1$. Note that the expression is $0$ unless $s_\nu=s_\rho$, so we consider only that case. ''Adding an edge at $x$" means adding
\begin{tikzpicture}[scale=.5]
 \node[int] (a) at (0,0) {};
 \node[int] (b) at (1,0) {};
 \draw (a) edge[->] (b);
 \node[above left] at (a) {$\scriptstyle x$};
 \node[above right] at (b) {$\scriptstyle v+1$};
\end{tikzpicture}
on the vertex $x$. Unless $x$ is an isolated vertex, it will cancel one term of the splitting. One can check that $d^2=0$ if and only if $s_\mu=-s_\rho$. Therefore, for $n\in\Z$ we can define the full graph complexes with a differential $d$:
\begin{equation*}
 \fGC_n = \prod_{v\geq 1} \prod_{e\geq 0} 
 \begin{cases}
  V_{v,e,0}^{+-+}[(n-1)e-n(v-1)] & \text{for $n$ even} \\ 
  V_{v,e,0}^{-+-}[(n-1)e-n(v-1)] & \text{for $n$ odd}
 \end{cases}.
\end{equation*}
The degree shifts are chosen such that there is a natural differential graded Lie algebra structure on $\fGC_n$, see \cite{grt}.

It can be seen that $d$ can not produce 1-valent and 2-valent vertices, tadpoles and multiple edges, if there were none of them before. Also, the differential obviously does not alter the number of connected components of a graph.
Therefore we may define several smaller subcomplexes.
Let $V^{even}_{v,e}:=V_{v,e,3}^{+-+}$, $V^{even*}_{v,e}:=V_{v,e,3}^{*-+}$, $V^{odd}_{v,e}:=V_{v,e,3}^{-+-}$ and $V^{odd*}_{v,e}:=V_{v,e,3}^{-*-}$. Let $\tilde V_{v,e,i}^{s_\nu s_\mu s_\rho}\subset V_{v,e,i}^{s_\nu s_\mu s_\rho}$ be the space generated by all connected graphs, and similarly let $\tilde V^{even}_{v,e}:=\tilde V_{v,e,3}^{+-+}$, $\tilde V^{even*}_{v,e}:=\tilde V_{v,e,3}^{*-+}$, $\tilde V^{odd}_{v,e}:=\tilde V_{v,e,3}^{-+-}$ and $\tilde V^{odd*}_{v,e}:=\tilde V_{v,e,3}^{-*-}$. Then we define the following subcomplexes of $\fGC_n$:

\begin{align*}
\fGC_{n}^{\geq 3} 
&=
\prod_{v\geq 1} \prod_{e\geq 0} 
\begin{cases}
  V^{even}_{v,e}[(n-1)e-n(v-1)] 
  &
  \text{for $n$ even}
  \\
   V^{odd}_{v,e}[(n-1)e-n(v-1)] 
  &
  \text{for $n$ odd}
\end{cases}
&
\\
\fGC_n^{\geq 3,\notadp} 
&= 
\prod_{v\geq 1} \prod_{e\geq 0} 
  V^{even*}_{v,e}[(n-1)e-n(v-1)] 
&  \text{for $n$ even}
   \\
\fGC_n^{\geq 3,\nomul} 
 &= \prod_{v\geq 1} \prod_{e\geq 0} 
   V^{odd*}_{v,e}[(n-1)e-n(v-1)] 
     &
   \text{for $n$ odd}
\\
\GC_n 
&=
\prod_{v\geq 1} \prod_{e\geq 0} 
\begin{cases}
  \tilde V^{even}_{v,e}[(n-1)e-n(v-1)] 
  &
  \text{for $n$ even}
  \\
    \tilde V^{odd}_{v,e}[(n-1)e-n(v-1)] 
  &
  \text{for $n$ odd}
\end{cases}
&
\\
\GC_n^{\notadp} 
&= 
\prod_{v\geq 1} \prod_{e\geq 0} 
  \tilde V^{even*}_{v,e}[(n-1)e-n(v-1)] 
&  \text{for $n$ even}
   \\
\GC_n^{\nomul} 
 &= \prod_{v\geq 1} \prod_{e\geq 0} 
   \tilde V^{odd*}_{v,e}[(n-1)e-n(v-1)] 
     &
   \text{for $n$ odd}
\end{align*}


In a connected graph a \emph{separating vertex} is a vertex which if removed makes the graph disconnected.
A graph without a separating vertex is called one-vertex irreducible.
The differential $d$ can not form a separating vertex, so there are subcomplexes generated by the one-vertex irreducible graphs, which we denote by adding a letter n (non-separable) to the name, e.g.\ $\GCn_n^{\nomul}$.

All complexes split into a product of subcomplexes according to the number of edges minus number of vertices $b=e-v$, which is preserved by the differential.\footnote{The number $b$ is of course minus the Euler characteristic of the graph as a topological space.} For them, the number $n$ (up to degree shift) does not really matter, only its parity. If we restrict to graphs with at least trivalent vertices those subcomplexes are finitely dimensional. Therefore, we define subcomplexes
$\GC_{n,b}\subset \GC_n$ etc. generated by graphs with fixed $e-v=b$. 

We denote the Euler characteristics of these finite dimensional complexes as follows:
\begin{equation}\label{equ:chidefs}
\begin{aligned}
\chi_b^{odd}&:=\sum_{v\geq 0}(-1)^v\dim\left(V^{odd}_{v,b+v}\right) = \chi(\fGC_{n,b}^{\geq 3}) & \text{for $n$ odd}
\\
\chi_b^{even}&:=\sum_{v\geq 0}(-1)^{b+v}\dim\left(V^{even}_{v,b+v}\right)=\chi(\fGC_{n,b}^{\geq 3}) & \text{for $n$ even}
\\
\chi_b^{odd*}&:=\sum_{v\geq 0}(-1)^v\dim\left(V^{odd*}_{v,b+v}\right) = \chi(\fGC_{n,b}^{\geq 3,\nomul}) & \text{for $n$ odd}
\\
\chi_b^{even*}&:=\sum_{v\geq 0}(-1)^{b+v}\dim\left(V^{even*}_{v,b+v}\right)=\chi(\fGC_{n,b}^{\geq 3,\notadp}) & \text{for $n$ even}
\\
\tilde \chi_b^{odd}&:=\sum_{v\geq 0}(-1)^v\dim\left(\tilde V^{odd}_{v,b+v}\right) = \chi(\GC_{n,b}^{}) & \text{for $n$ odd}
\\
\tilde \chi_b^{even}&:=\sum_{v\geq 0}(-1)^{b+v}\dim\left(\tilde V^{odd}_{v,b+v}\right) = \chi(\GC_{n,b}^{}) & \text{for $n$ even}
\\
\tilde \chi_b^{odd*}&:=\sum_{v\geq 0}(-1)^v\dim\left(\tilde V^{even*}_{v,b+v}\right) = \chi(\GC_{n,b}^{\nomul}) & \text{for $n$ odd}
\\
\tilde \chi_b^{even*}&:=\sum_{v\geq 0}(-1)^{b+v}\dim\left(\tilde V^{even*}_{v,b+v}\right) = \chi(\GC_{n,b}^{\notadp}) & \text{for $n$ even}
\end{aligned}
\end{equation}
In this paper we are interested in computing the above Euler characteristics. The numeric result is contained in Table \ref{tbl:eulerchar} below.

\section{Graphs with multiple edges may be omitted}\label{sec:multiple}

The cohomologies of the various graph complexes introduced above are highly related. Obviously, the graph complexes with disconnected graphs are just symmetric products of the complexes of connected graphs. Furthermore, it has been shown in \cite[Proposition 3.4]{grt} that adding the trivalence condition changes the cohomology of the graph complexes only by a list of known classes, and the omission of graphs with tadpoles does not change the cohomology further. 

Finally Conant and Vogtman \cite{conant_cut_2005} showed that the complexes of one-vertex irreducible graphs are quasi-isomorphic to their non-one-vertex irreducible relatives.


\begin{prop}[\cite{conant_cut_2005}, cf. Appendix F of \cite{grt}\footnote{ We note that in these references the result is only shown for one version of the graph complex. However, the proofs do not depend on the presence or absence of tadpoles or multiple edges. }
]
\label{Coh2}
\begin{align*}
 H(\GC_n)&=H(\GCn_n)
 &
 H(\GC_n^\nomul)&=H(\GCn_n^\nomul),
\end{align*}
\end{prop}

We can use these results to show Theorem \ref{thm:intromultiple} of the introduction.

\begin{proof}[Proof of Theorem \ref{thm:intromultiple}]
We actually prove that $H(\GCn_n^{})=H(\GCn_n^\nomul)\oplus\K[2n-3]$ and use Proposition \ref{Coh2}. We have splitting of complexes
\begin{equation*}
 \GCn_n^{}\cong\GCn_n'\oplus\K\Theta,
\end{equation*}
where $\Theta$ is the ``Theta" graph
\begin{tikzpicture}[baseline=-.65ex,scale=.5]
 \node[int] (a) at (0,0) {};
 \node[int] (b) at (1,0) {};
 \draw (a) edge (b);
 \draw (a) edge[bend left=35] (b);
 \draw (a) edge[bend right=35] (b);
\end{tikzpicture},
of degree $3-2n$ and $\GCn_n'$ is the remainder. It is clear that $H(\K\Theta)\cong\K[2n-3]$, so the claim reduces to showing that $H(\GCn_n')=H(\GCn_n^\nomul)$.

To be precise, let a \emph{multiple edge} $e$ be the set of all edges connecting the same pair of vertices. Let $N(e)$ be the number of edges in multiple edge $e$, let $S(e):=\left\lfloor\frac{N(e)}{2}\right\rfloor$  be its \emph{strength} and let the \emph{total strength} $S(\Gamma)$ be the sum of the strengths of all multiple edges in a graph $\Gamma$. The differential $d$ can not increase the total strength, so we have a filtration of $\GCn_n'$ by the total strength. The subcomplexes of fixed $b=e-v$ are finitely dimensional, so for all of them the spectral sequences converge to their cohomology, and therefore the original spectral sequence converges to $H(\GCn_n')$.

The differential $d^0$ on the first page does not decrease the total strength, or we can say
\begin{equation*}
 d^0\Gamma=\sum_{x\in V(\Gamma)}d^0_x\qquad\text{for}\qquad d^0_x\,``="\frac{1}{2}(\text{strength preserving splittings of $x$}) - (\text{adding an edge at $x$}),
\end{equation*}
where ''strength preserving splittings of $x$" are the splittings of $x$ that do not split multiple edge in two parts with odd number of edges.

Let a \emph{good} vertex be a trivalent vertex $x$ two of whose edges form a double edge:
\begin{tikzpicture}[baseline=-.65ex,scale=.5]
 \node[int] (a) at (0,0) {};
 \node[int] (b) at (1,0) {};
 \node[above] at (b) {$\scriptstyle x$};
 \node[int] (c) at (2,0) {};
 \draw (a) edge[bend left=20] (b);
 \draw (a) edge[bend right=20] (b);
 \draw (b) edge (c);
\end{tikzpicture}.
The other end of the double edge is denoted by $t(x)$ and the other end of the single edge is denoted by $s(x)$. We require $t(x)\neq s(x)$. For every good vertex $x$ we define a map $h_x$ such that locally:
\begin{equation}
\begin{tikzpicture}[baseline=-.65ex,scale=.8]
 \node[int] (a) at (0,0) {};
 \node[int] (b) at (1,0) {};
 \node[int] (c) at (2,0) {};
 \node[below] at (a) {$\scriptstyle t(x)$};
 \node[below] at (b) {$\scriptstyle x$};
 \node[below] at (c) {$\scriptstyle s(x)$};
 \draw (a) edge[bend left=20] (b);
 \draw (a) edge[bend right=20] (b);
 \draw (a) edge[very thick,bend left=50] node[above] {$\scriptstyle N$} (c);
 \draw (b) edge (c);
 \draw (a) edge (-.3,0.3);
 \draw (a) edge (-.4,0.1);
 \draw (a) edge (-.3,-0.3);
 \draw (a) edge (-.4,-0.1);
 \draw (c) edge (2.3,0.3);
 \draw (c) edge (2.4,0.1);
 \draw (c) edge (2.3,-0.3);
 \draw (c) edge (2.4,-0.1);
\end{tikzpicture}\qquad\mxto{h_x}\quad\pm
\begin{cases}
 \frac{1}{N+2} & \text{if $N$ odd} \\
 \frac{1}{N+1} & \text{if $N$ even}
\end{cases}\quad
\begin{tikzpicture}[baseline=-.65ex,scale=.8]
 \node[int] (a) at (0,0) {};
 \node[int] (b) at (1.5,0) {};
 \node[below] at (a) {$\scriptstyle t(x)$};
 \node[below] at (b) {$\scriptstyle s(x)$};
 \draw (a) edge[very thick] node[above] {$\scriptstyle N+2$} (b);
 \draw (a) edge (-.3,0.3);
 \draw (a) edge (-.4,0.1);
 \draw (a) edge (-.3,-0.3);
 \draw (a) edge (-.4,-0.1);
 \draw (b) edge (1.8,0.3);
 \draw (b) edge (1.9,0.1);
 \draw (b) edge (1.8,-0.3);
 \draw (b) edge (1.9,-0.1);
\end{tikzpicture},
\end{equation}
where the thick edge with number $N\geq 0$ indicates an $N$-fold edge, and the sign is chosen such that
\begin{equation*}
\begin{tikzpicture}[baseline=-.65ex,scale=.8]
 \node[int] (a) at (0,0) {};
 \node[int] (b) at (1,0) {};
 \node[int] (c) at (2,0) {};
 \node[below] at (a) {$\scriptstyle t(v)$};
 \node[below] at (b) {$\scriptstyle v$};
 \node[below] at (c) {$\scriptstyle s(v)$};
 \draw (a) edge[bend left=20] (b);
 \draw (a) edge[bend right=20] (b);
 \draw (a) edge[very thick,bend left=50] node[above] {$\scriptstyle N$} (c);
 \draw (b) edge[<-] (c);
 \draw (a) edge (-.3,0.3);
 \draw (a) edge (-.4,0.1);
 \draw (a) edge (-.3,-0.3);
 \draw (a) edge (-.4,-0.1);
 \draw (c) edge (2.3,0.3);
 \draw (c) edge (2.4,0.1);
 \draw (c) edge (2.3,-0.3);
 \draw (c) edge (2.4,-0.1);
\end{tikzpicture}\qquad\mxto{h_x}\quad +
\begin{cases}
 \frac{1}{N+2} & \text{if $N$ odd} \\
 \frac{1}{N+1} & \text{if $N$ even}
\end{cases}\quad
\begin{tikzpicture}[baseline=-.65ex,scale=.8]
 \node[int] (a) at (0,0) {};
 \node[int] (b) at (1.5,0) {};
 \node[below] at (a) {$\scriptstyle t(v)$};
 \node[below] at (b) {$\scriptstyle s(v)$};
 \draw (a) edge[very thick] node[above] {$\scriptstyle N+2$} (b);
 \draw (a) edge (-.3,0.3);
 \draw (a) edge (-.4,0.1);
 \draw (a) edge (-.3,-0.3);
 \draw (a) edge (-.4,-0.1);
 \draw (b) edge (1.8,0.3);
 \draw (b) edge (1.9,0.1);
 \draw (b) edge (1.8,-0.3);
 \draw (b) edge (1.9,-0.1);
\end{tikzpicture},
\end{equation*}
if $v$ is the last vertex, all other vertices keep their number and all edges keep their orientation. Note that $h_x$ does not change the total strength. We put $h_x=0$ if $x$ is not good, and
\begin{equation}
 h\Gamma=\sum_{x\in V(\Gamma)}h_x\Gamma.
\end{equation}

\begin{lemma}
It holds that
\begin{equation*}
 \left(d^0h+hd^0\right)\Gamma=2S(\Gamma)\Gamma.
\end{equation*}
\end{lemma}
\begin{proof}
We compute
\begin{equation*}
\left(d^0h+hd^0\right)\Gamma = \sum_{y\in V(h(\Gamma))} d_y^0 \sum_{x\in V(\Gamma)} h_x \Gamma + \sum_{x\in V(d^0(\Gamma))} h_x \sum_{y\in V(\Gamma)} d_y^0 \Gamma =
\sum_{\substack{x,y\in V(\Gamma)\\x\neq y}}\left(d^0_yh_x\Gamma+h_xd^0_y\Gamma\right) + 2\sum_{x\in V(\Gamma)}h_xd^0_x\Gamma.
\end{equation*}
We claim that $d^0_y$ can not change the property of being good of a vertex $x\neq y$. Clearly, $d^0_y$ can not change a good vertex to become not good. On the other hand, it can not affect non-neighbors, can not change the valence of other vertices and can not produce multiple edge. Therefore, if $d^0_y$ makes $x$ good, $x$ was already trivalent with two of the edges pointing to the same vertex before acting of $d^0_y$. The only possibility when $x$ was not good before acting is when $x$ was a trivalent neighbour of $y$ all of whose three edges form a triple edge towards $y$, that is
\begin{tikzpicture}[baseline=-.65ex,scale=.5]
 \node[int] (a) at (0,0) {};
 \node[int] (b) at (1,0) {};
 \node[below] at (a) {$\scriptstyle y$};
 \node[below] at (b) {$\scriptstyle x$};
 \draw (a) edge (b);
 \draw (a) edge[bend left=35] (b);
 \draw (a) edge[bend right=35] (b);
 \draw (a) edge (-.3,0.3);
 \draw (a) edge (-.4,0.1);
 \draw (a) edge (-.3,-0.3);
 \draw (a) edge (-.4,-0.1);
\end{tikzpicture}.
But $y$ can not be a separating vertex and hence can not be connected to anything else than $x$, and the whole graph is the theta graph $\Theta_n$. But this graph has been explicitly excluded. Therefore
\begin{multline}
\label{deco}
\left(d^0h+hd^0\right)\Gamma =
\sum_{\substack{x\in V(\Gamma)\\x\text{ good}}}\sum_{\substack{y\in V(\Gamma) \\ y\neq x}} \left(d^0_yh_x\Gamma+h_xd^0_y\Gamma\right) + 2\sum_{x\in V(\Gamma)}h_xd^0_x\Gamma = \\
= \sum_{\substack{x\in V(\Gamma)\\x\text{ good}}}\sum_{\substack{y\in V(\Gamma) \\ y\notin\{x,t(x),s(x)\} }} \left(d^0_yh_x\Gamma+h_xd^0_y\Gamma\right) + 
\sum_{\substack{x\in V(\Gamma)\\x\text{ good}}}\left(d^0_{t(x)}h_x\Gamma+h_xd^0_{t(x)}\Gamma\right) + 
\sum_{\substack{x\in V(\Gamma)\\x\text{ good}}}\left(d^0_{s(x)}h_x\Gamma+h_xd^0_{s(x)}\Gamma\right) + 2\sum_{x\in V(\Gamma)}h_xd^0_x\Gamma.
\end{multline}
The first term is trivially zero. We claim that the second term is also zero. It is enough to assume that $x$ is the last vertex $v$. We consider separately the cases of odd and even numbers of ``bridging'' vertices between $s(v)$ and $t(v)$.  First, in the even case:
\begin{equation*}
\begin{tikzpicture}[baseline=-.65ex,scale=.8]
 \node[int] (a) at (0,0) {};
 \node[int] (b) at (1,0) {};
 \node[int] (c) at (2,0) {};
 \node[below] at (a) {$\scriptstyle t(v)$};
 \node[below] at (b) {$\scriptstyle v$};
 \node[below] at (c) {$\scriptstyle s(v)$};
 \draw (a) edge[bend left=20] (b);
 \draw (a) edge[bend right=20] (b);
 \draw (a) edge[very thick,bend left=50] node[above] {$\scriptstyle 2N$} (c);
 \draw (b) edge[<-] (c);
 \draw (a) edge (-.3,0.3);
 \draw (a) edge (-.4,0.1);
 \draw (a) edge (-.3,-0.3);
 \draw (a) edge (-.4,-0.1);
 \draw (c) edge (2.3,0.3);
 \draw (c) edge (2.4,0.1);
 \draw (c) edge (2.3,-0.3);
 \draw (c) edge (2.4,-0.1);
\end{tikzpicture}\quad\mxto{h_v}\quad\frac{1}{2N+1}\quad
\begin{tikzpicture}[baseline=-.65ex,scale=.8]
 \node[int] (a) at (0,0) {};
 \node[int] (b) at (1.5,0) {};
 \node[below] at (a) {$\scriptstyle t(v)$};
 \node[below] at (b) {$\scriptstyle s(v)$};
 \draw (a) edge[very thick] node[above] {$\scriptstyle 2N+2$} (b);
 \draw (a) edge (-.3,0.3);
 \draw (a) edge (-.4,0.1);
 \draw (a) edge (-.3,-0.3);
 \draw (a) edge (-.4,-0.1);
 \draw (b) edge (1.8,0.3);
 \draw (b) edge (1.9,0.1);
 \draw (b) edge (1.8,-0.3);
 \draw (b) edge (1.9,-0.1);
\end{tikzpicture}\quad\mxto{d^0_{t(v)}}\quad\frac{1}{4N+2}\sum_{k=0}^{N+1}\binom{2N+2}{2k}\sum_{conn}\;
\begin{tikzpicture}[baseline=1ex,scale=.8]
 \node[int] (a) at (0,0) {};
 \node[int] (b) at (2,0) {};
 \node[int] (d) at (0,1) {};
 \node[below] at (a) {$\scriptstyle t(v)$};
 \node[above] at (b) {$\scriptstyle s(v)$};
 \node[above] at (d) {$\scriptstyle v$};
 \draw (a) edge[->] (d);
 \draw (a) edge[very thick] node[below] {$\scriptstyle 2N+2-2k$} (b);
 \draw (d) edge[very thick] node[above] {$\scriptstyle 2k$} (b);
 \draw (a) edge (-.4,0.1);
 \draw (a) edge (-.4,-0.1);
 \draw (d) edge (-.4,1.1);
 \draw (d) edge (-.4,0.9);
 \draw (b) edge (2.3,0.3);
 \draw (b) edge (2.4,0.1);
 \draw (b) edge (2.3,-0.3);
 \draw (b) edge (2.4,-0.1);
\end{tikzpicture},
\end{equation*}
\begin{multline*}
\begin{tikzpicture}[baseline=-.65ex,scale=.8]
 \node[int] (a) at (0,0) {};
 \node[int] (b) at (1,0) {};
 \node[int] (c) at (2,0) {};
 \node[below] at (a) {$\scriptstyle t(v)$};
 \node[below] at (b) {$\scriptstyle v$};
 \node[below] at (c) {$\scriptstyle s(v)$};
 \draw (a) edge[bend left=20] (b);
 \draw (a) edge[bend right=20] (b);
 \draw (a) edge[very thick,bend left=50] node[above] {$\scriptstyle 2N$} (c);
 \draw (b) edge[<-] (c);
 \draw (a) edge (-.3,0.3);
 \draw (a) edge (-.4,0.1);
 \draw (a) edge (-.3,-0.3);
 \draw (a) edge (-.4,-0.1);
 \draw (c) edge (2.3,0.3);
 \draw (c) edge (2.4,0.1);
 \draw (c) edge (2.3,-0.3);
 \draw (c) edge (2.4,-0.1);
\end{tikzpicture}\quad\mxto{d^0_{t(v)}}\quad\sum_{k=0}^N\binom{2N}{2k}
\sum_{conn}\;
\begin{tikzpicture}[baseline=1ex,scale=.8]
 \node[int] (a) at (0,0) {};
 \node[int] (b) at (1,0) {};
 \node[int] (c) at (2,0) {};
 \node[int] (d) at (0,1) {};
 \node[below] at (a) {$\scriptstyle t(v)$};
 \node[below] at (b) {$\scriptstyle v$};
 \node[above] at (c) {$\scriptstyle s(v)$};
 \node[above] at (d) {$\scriptstyle v+1$};
 \draw (a) edge[bend left=20] (b);
 \draw (a) edge[bend right=20] (b);
 \draw (a) edge[->] (d);
 \draw (a) edge[very thick,bend right=50] node[below] {$\scriptstyle 2N-2k$} (c);
 \draw (d) edge[very thick] node[above] {$\scriptstyle 2k$} (c);
 \draw (b) edge[<-] (c);
 \draw (a) edge (-.4,0.1);
 \draw (a) edge (-.4,-0.1);
 \draw (d) edge (-.4,1.1);
 \draw (d) edge (-.4,0.9);
 \draw (c) edge (2.3,0.3);
 \draw (c) edge (2.4,0.1);
 \draw (c) edge (2.3,-0.3);
 \draw (c) edge (2.4,-0.1);
\end{tikzpicture}\quad\mxto{h_v}\;\;-\sum_{k=0}^N\binom{2N}{2k}\frac{1}{2N-2k+1}\sum_{conn}\;
\begin{tikzpicture}[baseline=1ex,scale=.8]
 \node[int] (a) at (0,0) {};
 \node[int] (b) at (2,0) {};
 \node[int] (d) at (0,1) {};
 \node[below] at (a) {$\scriptstyle t(v)$};
 \node[above] at (b) {$\scriptstyle s(v)$};
 \node[above] at (d) {$\scriptstyle v$};
 \draw (a) edge[->] (d);
 \draw (a) edge[very thick] node[below] {$\scriptstyle 2N-2k+2$} (b);
 \draw (d) edge[very thick] node[above] {$\scriptstyle 2k$} (b);
 \draw (a) edge (-.4,0.1);
 \draw (a) edge (-.4,-0.1);
 \draw (d) edge (-.4,1.1);
 \draw (d) edge (-.4,0.9);
 \draw (b) edge (2.3,0.3);
 \draw (b) edge (2.4,0.1);
 \draw (b) edge (2.3,-0.3);
 \draw (b) edge (2.4,-0.1);
\end{tikzpicture}\;=\\
=-\frac{1}{2}\left[\sum_{k=0}^{N+1}\binom{2N}{2k}\frac{1}{2N-2k+1}\sum_{conn}\;
\begin{tikzpicture}[baseline=1ex,scale=.8]
 \node[int] (a) at (0,0) {};
 \node[int] (b) at (2,0) {};
 \node[int] (d) at (0,1) {};
 \node[below] at (a) {$\scriptstyle t(v)$};
 \node[above] at (b) {$\scriptstyle s(v)$};
 \node[above] at (d) {$\scriptstyle v$};
 \draw (a) edge[->] (d);
 \draw (a) edge[very thick] node[below] {$\scriptstyle 2N-2k+2$} (b);
 \draw (d) edge[very thick] node[above] {$\scriptstyle 2k$} (b);
 \draw (a) edge (-.4,0.1);
 \draw (a) edge (-.4,-0.1);
 \draw (d) edge (-.4,1.1);
 \draw (d) edge (-.4,0.9);
 \draw (b) edge (2.3,0.3);
 \draw (b) edge (2.4,0.1);
 \draw (b) edge (2.3,-0.3);
 \draw (b) edge (2.4,-0.1);
\end{tikzpicture}\;
+\sum_{k=0}^{N+1}\binom{2N}{2k}\frac{1}{2N-2k+1}\sum_{conn}\;
\begin{tikzpicture}[baseline=1ex,scale=.8]
 \node[int] (a) at (0,0) {};
 \node[int] (b) at (2,0) {};
 \node[int] (d) at (0,1) {};
 \node[below] at (a) {$\scriptstyle v$};
 \node[above] at (b) {$\scriptstyle s(v)$};
 \node[above] at (d) {$\scriptstyle t(v)$};
 \draw (a) edge[<-] (d);
 \draw (a) edge[very thick] node[below] {$\scriptstyle 2N-2k+2$} (b);
 \draw (d) edge[very thick] node[above] {$\scriptstyle 2k$} (b);
 \draw (a) edge (-.4,0.1);
 \draw (a) edge (-.4,-0.1);
 \draw (d) edge (-.4,1.1);
 \draw (d) edge (-.4,0.9);
 \draw (b) edge (2.3,0.3);
 \draw (b) edge (2.4,0.1);
 \draw (b) edge (2.3,-0.3);
 \draw (b) edge (2.4,-0.1);
\end{tikzpicture}\right]=\\
=-\frac{1}{2}\sum_{k=0}^{N+1}\left[\binom{2N}{2k}\frac{1}{2N-2k+1}+\binom{2N}{2N-2k+2}\frac{1}{2k-1}\right]\sum_{conn}\;
\begin{tikzpicture}[baseline=1ex,scale=.8]
 \node[int] (a) at (0,0) {};
 \node[int] (b) at (2,0) {};
 \node[int] (d) at (0,1) {};
 \node[below] at (a) {$\scriptstyle t(v)$};
 \node[above] at (b) {$\scriptstyle s(v)$};
 \node[above] at (d) {$\scriptstyle v$};
 \draw (a) edge[->] (d);
 \draw (a) edge[very thick] node[below] {$\scriptstyle 2N-2k+2$} (b);
 \draw (d) edge[very thick] node[above] {$\scriptstyle 2k$} (b);
 \draw (a) edge (-.4,0.1);
 \draw (a) edge (-.4,-0.1);
 \draw (d) edge (-.4,1.1);
 \draw (d) edge (-.4,0.9);
 \draw (b) edge (2.3,0.3);
 \draw (b) edge (2.4,0.1);
 \draw (b) edge (2.3,-0.3);
 \draw (b) edge (2.4,-0.1);
\end{tikzpicture}\;=\\
=\frac{-1}{4N+2}\sum_{k=0}^{N+1}\binom{2N+2}{2k}\sum_{conn}\;
\begin{tikzpicture}[baseline=1ex,scale=.8]
 \node[int] (a) at (0,0) {};
 \node[int] (b) at (2,0) {};
 \node[int] (d) at (0,1) {};
 \node[below] at (a) {$\scriptstyle t(v)$};
 \node[above] at (b) {$\scriptstyle s(v)$};
 \node[above] at (d) {$\scriptstyle v$};
 \draw (a) edge[->] (d);
 \draw (a) edge[very thick] node[below] {$\scriptstyle 2N+2-2k$} (b);
 \draw (d) edge[very thick] node[above] {$\scriptstyle 2k$} (b);
 \draw (a) edge (-.4,0.1);
 \draw (a) edge (-.4,-0.1);
 \draw (d) edge (-.4,1.1);
 \draw (d) edge (-.4,0.9);
 \draw (b) edge (2.3,0.3);
 \draw (b) edge (2.4,0.1);
 \draw (b) edge (2.3,-0.3);
 \draw (b) edge (2.4,-0.1);
\end{tikzpicture},
\end{multline*}
where $\sum_{conn}$ is the sum over all possibilities of connecting remaining edges of $v(t)$ to new vertices. We have omitted the term of ''Adding an edge" at $t(v)$ in the action of $d^0_{t(v)}$, but it trivially cancels. For an odd number of ``bridging'' edges the situation is similar:
\begin{equation*}
\begin{tikzpicture}[baseline=-.65ex,scale=.8]
 \node[int] (a) at (0,0) {};
 \node[int] (b) at (1,0) {};
 \node[int] (c) at (2,0) {};
 \node[below] at (a) {$\scriptstyle t(v)$};
 \node[below] at (b) {$\scriptstyle v$};
 \node[below] at (c) {$\scriptstyle s(v)$};
 \draw (a) edge[bend left=20] (b);
 \draw (a) edge[bend right=20] (b);
 \draw (a) edge[very thick,bend left=50] node[above] {$\scriptstyle 2N-1$} (c);
 \draw (b) edge[<-] (c);
 \draw (a) edge (-.3,0.3);
 \draw (a) edge (-.4,0.1);
 \draw (a) edge (-.3,-0.3);
 \draw (a) edge (-.4,-0.1);
 \draw (c) edge (2.3,0.3);
 \draw (c) edge (2.4,0.1);
 \draw (c) edge (2.3,-0.3);
 \draw (c) edge (2.4,-0.1);
\end{tikzpicture}\quad\mxto{h_v}\quad\frac{1}{2N+1}\quad
\begin{tikzpicture}[baseline=-.65ex,scale=.8]
 \node[int] (a) at (0,0) {};
 \node[int] (b) at (1.5,0) {};
 \node[below] at (a) {$\scriptstyle t(v)$};
 \node[below] at (b) {$\scriptstyle s(v)$};
 \draw (a) edge[very thick] node[above] {$\scriptstyle 2N+1$} (b);
 \draw (a) edge (-.3,0.3);
 \draw (a) edge (-.4,0.1);
 \draw (a) edge (-.3,-0.3);
 \draw (a) edge (-.4,-0.1);
 \draw (b) edge (1.8,0.3);
 \draw (b) edge (1.9,0.1);
 \draw (b) edge (1.8,-0.3);
 \draw (b) edge (1.9,-0.1);
\end{tikzpicture}\quad\mxto{d^0_{t(v)}}\quad\frac{1}{4N+2}\sum_{k=0}^{2N+1}\binom{2N+1}{k}\sum_{conn}\;
\begin{tikzpicture}[baseline=1ex,scale=.8]
 \node[int] (a) at (0,0) {};
 \node[int] (b) at (2,0) {};
 \node[int] (d) at (0,1) {};
 \node[below] at (a) {$\scriptstyle t(v)$};
 \node[above] at (b) {$\scriptstyle s(v)$};
 \node[above] at (d) {$\scriptstyle v$};
 \draw (a) edge[->] (d);
 \draw (a) edge[very thick] node[below] {$\scriptstyle 2N+1-k$} (b);
 \draw (d) edge[very thick] node[above] {$\scriptstyle k$} (b);
 \draw (a) edge (-.4,0.1);
 \draw (a) edge (-.4,-0.1);
 \draw (d) edge (-.4,1.1);
 \draw (d) edge (-.4,0.9);
 \draw (b) edge (2.3,0.3);
 \draw (b) edge (2.4,0.1);
 \draw (b) edge (2.3,-0.3);
 \draw (b) edge (2.4,-0.1);
\end{tikzpicture},
\end{equation*}
\begin{multline*}
\begin{tikzpicture}[baseline=-.65ex,scale=.8]
 \node[int] (a) at (0,0) {};
 \node[int] (b) at (1,0) {};
 \node[int] (c) at (2,0) {};
 \node[below] at (a) {$\scriptstyle t(v)$};
 \node[below] at (b) {$\scriptstyle v$};
 \node[below] at (c) {$\scriptstyle s(v)$};
 \draw (a) edge[bend left=20] (b);
 \draw (a) edge[bend right=20] (b);
 \draw (a) edge[very thick,bend left=50] node[above] {$\scriptstyle 2N-1$} (c);
 \draw (b) edge[<-] (c);
 \draw (a) edge (-.3,0.3);
 \draw (a) edge (-.4,0.1);
 \draw (a) edge (-.3,-0.3);
 \draw (a) edge (-.4,-0.1);
 \draw (c) edge (2.3,0.3);
 \draw (c) edge (2.4,0.1);
 \draw (c) edge (2.3,-0.3);
 \draw (c) edge (2.4,-0.1);
\end{tikzpicture}\quad\mxto{d^0_{t(v)}}\quad\sum_{k=0}^{2N-1}\binom{2N-1}{k}\sum_{conn}\;
\begin{tikzpicture}[baseline=1ex,scale=.8]
 \node[int] (a) at (0,0) {};
 \node[int] (b) at (1,0) {};
 \node[int] (c) at (2,0) {};
 \node[int] (d) at (0,1) {};
 \node[below] at (a) {$\scriptstyle t(v)$};
 \node[below] at (b) {$\scriptstyle v$};
 \node[above] at (c) {$\scriptstyle s(v)$};
 \node[above] at (d) {$\scriptstyle v+1$};
 \draw (a) edge[bend left=20] (b);
 \draw (a) edge[bend right=20] (b);
 \draw (a) edge[->] (d);
 \draw (a) edge[very thick,bend right=50] node[below] {$\scriptstyle 2N-1-k$} (c);
 \draw (d) edge[very thick] node[above] {$\scriptstyle k$} (c);
 \draw (b) edge[<-] (c);
 \draw (a) edge (-.4,0.1);
 \draw (a) edge (-.4,-0.1);
 \draw (d) edge (-.4,1.1);
 \draw (d) edge (-.4,0.9);
 \draw (c) edge (2.3,0.3);
 \draw (c) edge (2.4,0.1);
 \draw (c) edge (2.3,-0.3);
 \draw (c) edge (2.4,-0.1);
\end{tikzpicture}\quad\mxto{h_v}\\
\mxto{h_v}\quad
-\sum_{k=0}^{N-1}\binom{2N-1}{2k}\frac{1}{2N-2k+1}\sum_{conn}\;
\begin{tikzpicture}[baseline=1ex,scale=.8]
 \node[int] (a) at (0,0) {};
 \node[int] (b) at (2,0) {};
 \node[int] (d) at (0,1) {};
 \node[below] at (a) {$\scriptstyle t(v)$};
 \node[above] at (b) {$\scriptstyle s(v)$};
 \node[above] at (d) {$\scriptstyle v$};
 \draw (a) edge[->] (d);
 \draw (a) edge[very thick] node[below] {$\scriptstyle 2N-2k+1$} (b);
 \draw (d) edge[very thick] node[above] {$\scriptstyle 2k$} (b);
 \draw (a) edge (-.4,0.1);
 \draw (a) edge (-.4,-0.1);
 \draw (d) edge (-.4,1.1);
 \draw (d) edge (-.4,0.9);
 \draw (b) edge (2.3,0.3);
 \draw (b) edge (2.4,0.1);
 \draw (b) edge (2.3,-0.3);
 \draw (b) edge (2.4,-0.1);
\end{tikzpicture}\;\;-\sum_{k=0}^{N-1}\binom{2N-1}{2k+1}\frac{1}{2N-2k-1}\sum_{conn}\;
\begin{tikzpicture}[baseline=1ex,scale=.8]
 \node[int] (a) at (0,0) {};
 \node[int] (b) at (2,0) {};
 \node[int] (d) at (0,1) {};
 \node[below] at (a) {$\scriptstyle t(v)$};
 \node[above] at (b) {$\scriptstyle s(v)$};
 \node[above] at (d) {$\scriptstyle v$};
 \draw (a) edge[->] (d);
 \draw (a) edge[very thick] node[below] {$\scriptstyle 2N-2k$} (b);
 \draw (d) edge[very thick] node[above] {$\scriptstyle 2k+1$} (b);
 \draw (a) edge (-.4,0.1);
 \draw (a) edge (-.4,-0.1);
 \draw (d) edge (-.4,1.1);
 \draw (d) edge (-.4,0.9);
 \draw (b) edge (2.3,0.3);
 \draw (b) edge (2.4,0.1);
 \draw (b) edge (2.3,-0.3);
 \draw (b) edge (2.4,-0.1);
\end{tikzpicture}\;=\\
=-\sum_{k=0}^N\left[\binom{2N-1}{2k}\frac{1}{2N-2k+1}+\binom{2N-1}{2N-2k+1}\frac{1}{2k-1}\right]\sum_{conn}\;
\begin{tikzpicture}[baseline=1ex,scale=.8]
 \node[int] (a) at (0,0) {};
 \node[int] (b) at (2,0) {};
 \node[int] (d) at (0,1) {};
 \node[below] at (a) {$\scriptstyle t(v)$};
 \node[above] at (b) {$\scriptstyle s(v)$};
 \node[above] at (d) {$\scriptstyle v$};
 \draw (a) edge[->] (d);
 \draw (a) edge[very thick] node[below] {$\scriptstyle 2N-2k+1$} (b);
 \draw (d) edge[very thick] node[above] {$\scriptstyle 2k$} (b);
 \draw (a) edge (-.4,0.1);
 \draw (a) edge (-.4,-0.1);
 \draw (d) edge (-.4,1.1);
 \draw (d) edge (-.4,0.9);
 \draw (b) edge (2.3,0.3);
 \draw (b) edge (2.4,0.1);
 \draw (b) edge (2.3,-0.3);
 \draw (b) edge (2.4,-0.1);
\end{tikzpicture}\;=\\
=\frac{-1}{2N+1}\sum_{k=0}^N\binom{2N+1}{2k}
\sum_{conn}\;
\begin{tikzpicture}[baseline=1ex,scale=.8]
 \node[int] (a) at (0,0) {};
 \node[int] (b) at (2,0) {};
 \node[int] (d) at (0,1) {};
 \node[below] at (a) {$\scriptstyle t(v)$};
 \node[above] at (b) {$\scriptstyle s(v)$};
 \node[above] at (d) {$\scriptstyle v$};
 \draw (a) edge[->] (d);
 \draw (a) edge[very thick] node[below] {$\scriptstyle 2N-2k+1$} (b);
 \draw (d) edge[very thick] node[above] {$\scriptstyle 2k$} (b);
 \draw (a) edge (-.4,0.1);
 \draw (a) edge (-.4,-0.1);
 \draw (d) edge (-.4,1.1);
 \draw (d) edge (-.4,0.9);
 \draw (b) edge (2.3,0.3);
 \draw (b) edge (2.4,0.1);
 \draw (b) edge (2.3,-0.3);
 \draw (b) edge (2.4,-0.1);
\end{tikzpicture}\;=
\frac{-1}{4N+2}\sum_{k=0}^{2N+1}\binom{2N+1}{k}\sum_{conn}\;
\begin{tikzpicture}[baseline=1ex,scale=.8]
 \node[int] (a) at (0,0) {};
 \node[int] (b) at (2,0) {};
 \node[int] (d) at (0,1) {};
 \node[below] at (a) {$\scriptstyle t(v)$};
 \node[above] at (b) {$\scriptstyle s(v)$};
 \node[above] at (d) {$\scriptstyle v$};
 \draw (a) edge[->] (d);
 \draw (a) edge[very thick] node[below] {$\scriptstyle 2N+1-k$} (b);
 \draw (d) edge[very thick] node[above] {$\scriptstyle k$} (b);
 \draw (a) edge (-.4,0.1);
 \draw (a) edge (-.4,-0.1);
 \draw (d) edge (-.4,1.1);
 \draw (d) edge (-.4,0.9);
 \draw (b) edge (2.3,0.3);
 \draw (b) edge (2.4,0.1);
 \draw (b) edge (2.3,-0.3);
 \draw (b) edge (2.4,-0.1);
\end{tikzpicture}.
\end{multline*}
Interchanging
\begin{tikzpicture}[baseline=-.65ex,scale=.5]
 \node[int] (a) at (0,0) {};
 \node[int] (b) at (1,0) {};
 \node[int] (c) at (2,0) {};
 \draw (a) edge[bend left=20] (b);
 \draw (a) edge[bend right=20] (b);
 \draw (b) edge[<-] (c);
\end{tikzpicture} with 
\begin{tikzpicture}[baseline=-.65ex,scale=.5]
 \node[int] (a) at (2,0) {};
 \node[int] (b) at (1,0) {};
 \node[int] (c) at (0,0) {};
 \draw (a) edge[bend left=20] (b);
 \draw (a) edge[bend right=20] (b);
 \draw (b) edge[<-] (c);
\end{tikzpicture},
$s(v)$ with $t(v)$ and vice versa in the above calculations leads to the conclusion that the third term in \eqref{deco} is $d^0_xh_x\Gamma+h_xd^0_{s(x)}\Gamma=0$.
The remaining term is:
\begin{equation*}
\left(d^0h+hd^0\right)\Gamma = 2\sum_{x\in V(\Gamma)}h_xd^0_x\Gamma = 2\sum_{x\in V(\Gamma)}h_{v+1}d^0_x\Gamma.
\end{equation*}
It suffices to consider terms for which $d^0_x$ makes the new vertex $v+1$ good, otherwise the term is zero. Vertex $v+1$ in $d^0_x\Gamma$ has a single edge towards $x$, so it is good if and only if a multiple edge $e$ with $N(e)\geq 2$ has been split into a double edge heading towards $v+1$ and an $(N(e)-2)$-fold edge heading towards the new $x$, and all other edges heading towards the new $x$, i.~e.\
\begin{equation*}
\begin{tikzpicture}[baseline=-.65ex,scale=.8]
 \node[int] (b) at (0,0) {};
 \node[int] (a) at (1.5,0) {};
 \node[below] at (b) {$\scriptstyle x$};
 \draw (a) edge[very thick] node[above] {$\scriptstyle N$} (b);
 \draw (b) edge (-.3,0.3);
 \draw (b) edge (-.4,0.1);
 \draw (b) edge (-.3,-0.3);
 \draw (b) edge (-.4,-0.1);
 \draw (a) edge (1.8,0.3);
 \draw (a) edge (1.9,0.1);
 \draw (a) edge (1.8,-0.3);
 \draw (a) edge (1.9,-0.1);
\end{tikzpicture}\quad\longmapsto\quad\frac{1}{2}\binom{N}{2}\quad
\begin{tikzpicture}[baseline=-.65ex,scale=.8]
 \node[int] (c) at (0,0) {};
 \node[int] (b) at (1,0) {};
 \node[int] (a) at (2,0) {};
 \node[below] at (c) {$\scriptstyle x$};
 \node[below] at (b) {$\scriptstyle v+1$};
 \draw (a) edge[bend left=20] (b);
 \draw (a) edge[bend right=20] (b);
 \draw (a) edge[very thick,bend right=50] node[above] {$\scriptstyle N-2$} (c);
 \draw (b) edge[<-] (c);
 \draw (c) edge (-.3,0.3);
 \draw (c) edge (-.4,0.1);
 \draw (c) edge (-.3,-0.3);
 \draw (c) edge (-.4,-0.1);
 \draw (a) edge (2.3,0.3);
 \draw (a) edge (2.4,0.1);
 \draw (a) edge (2.3,-0.3);
 \draw (a) edge (2.4,-0.1);
\end{tikzpicture}\, .
\end{equation*}
We have to check that this term is not cancelled by ''Adding an edge", i.e.\ that it is not at the same time true that $N=2$ and that there are no other edges towards $x$. But in that case, since the other end of the $N$-fold edge is not separating, the whole graph would be
$\begin{tikzpicture}[baseline=-.65ex,scale=.5]
 \node[int] (a) at (0,0) {};
 \node[int] (b) at (1,0) {};
 \draw (a) edge[bend left=20] (b);
 \draw (a) edge[bend right=20] (b);
\end{tikzpicture}=0$. Therefore
\begin{multline*}
\begin{tikzpicture}[baseline=-.65ex,scale=.8]
 \node[int] (a) at (0,0) {};
 \node[int] (b) at (1,0) {};
 \node[int] (c) at (.6,.8) {};
 \node[int] (d) at (.6,-.8) {};
 \draw (a) edge[very thick] (b);
 \draw (a) edge[very thick] (c);
 \draw (a) edge[very thick] (d);
 \node[below] at (a) {$\scriptstyle x$};
 \draw (a) edge (-.3,0.3);
 \draw (a) edge (-.4,0.1);
 \draw (a) edge (-.3,-0.3);
 \draw (a) edge (-.4,-0.1);
 \draw (b) edge (1.3,0.3);
 \draw (b) edge (1.4,0.1);
 \draw (b) edge (1.3,-0.3);
 \draw (b) edge (1.4,-0.1);
 \draw (c) edge (.6,1.2);
 \draw (c) edge (.8,1.16);
 \draw (c) edge (.95,1.05);
 \draw (c) edge (1,.85);
 \draw (d) edge (.6,-1.2);
 \draw (d) edge (.8,-1.16);
 \draw (d) edge (.95,-1.05);
 \draw (d) edge (1,-.85);
\end{tikzpicture}\quad\mxto{d^0_x}\sum_{e \text{ multiple edge at }x} \frac{1}{2}\binom{N(e)}{2}\quad
\begin{tikzpicture}[baseline=-.65ex,scale=.8]
 \node[int] (c) at (0,0) {};
 \node[int] (b) at (1,0) {};
 \node[int] (a) at (2,0) {};
 \node[int] (e) at (.2,.9) {};
 \node[int] (f) at (.2,-.9) {};
 \draw (c) edge[very thick] (e);
 \draw (c) edge[very thick] (f);
 \node[below right] at (c) {$\scriptstyle x$};
 \node[below] at (b) {$\scriptstyle v+1$};
 \draw (a) edge[bend left=20] (b);
 \draw (a) edge[bend right=20] (b);
 \draw (a) edge[very thick,bend right=50] node[above] {$\scriptstyle N(e)-2$} (c);
 \draw (b) edge[<-] (c);
 \draw (c) edge (-.3,0.3);
 \draw (c) edge (-.4,0.1);
 \draw (c) edge (-.3,-0.3);
 \draw (c) edge (-.4,-0.1);
 \draw (a) edge (2.3,0.3);
 \draw (a) edge (2.4,0.1);
 \draw (a) edge (2.3,-0.3);
 \draw (a) edge (2.4,-0.1);
 \draw (e) edge (0,1.25);
 \draw (e) edge (.2,1.3);
 \draw (e) edge (.4,1.25);
 \draw (e) edge (.55,1.15);
 \draw (f) edge (0,-1.25);
 \draw (f) edge (.2,-1.3);
 \draw (f) edge (.4,-1.25);
 \draw (f) edge (.55,-1.15);
\end{tikzpicture}\;+\text{(something where $v+1$ is not good)}\;\mxto{h_{v+1}}\\
\mxto{h_{v+1}}\sum_{e \text{ multiple edge at }x} \frac{1}{2}\binom{N(e)}{2}
\begin{cases}
 \frac{1}{N(e)} & \text{if $N$ odd} \\
 \frac{1}{N(e)-1} & \text{if $N$ even}
\end{cases}\quad
\begin{tikzpicture}[baseline=-.65ex,scale=.8]
 \node[int] (a) at (0,0) {};
 \node[int] (b) at (1,0) {};
 \node[int] (c) at (.6,.8) {};
 \node[int] (d) at (.6,-.8) {};
 \draw (a) edge[very thick] (b);
 \draw (a) edge[very thick] (c);
 \draw (a) edge[very thick] (d);
 \node[below] at (a) {$\scriptstyle x$};
 \draw (a) edge (-.3,0.3);
 \draw (a) edge (-.4,0.1);
 \draw (a) edge (-.3,-0.3);
 \draw (a) edge (-.4,-0.1);
 \draw (b) edge (1.3,0.3);
 \draw (b) edge (1.4,0.1);
 \draw (b) edge (1.3,-0.3);
 \draw (b) edge (1.4,-0.1);
 \draw (c) edge (.6,1.2);
 \draw (c) edge (.8,1.16);
 \draw (c) edge (.95,1.05);
 \draw (c) edge (1,.85);
 \draw (d) edge (.6,-1.2);
 \draw (d) edge (.8,-1.16);
 \draw (d) edge (.95,-1.05);
 \draw (d) edge (1,-.85);
\end{tikzpicture}\; \;= \;\;\frac{1}{2}
\sum_{e \text{ multiple edge at }x}S(e)\quad
\begin{tikzpicture}[baseline=-.65ex,scale=.8]
 \node[int] (a) at (0,0) {};
 \node[int] (b) at (1,0) {};
 \node[int] (c) at (.6,.8) {};
 \node[int] (d) at (.6,-.8) {};
 \draw (a) edge[very thick] (b);
 \draw (a) edge[very thick] (c);
 \draw (a) edge[very thick] (d);
 \node[below] at (a) {$\scriptstyle x$};
 \draw (a) edge (-.3,0.3);
 \draw (a) edge (-.4,0.1);
 \draw (a) edge (-.3,-0.3);
 \draw (a) edge (-.4,-0.1);
 \draw (b) edge (1.3,0.3);
 \draw (b) edge (1.4,0.1);
 \draw (b) edge (1.3,-0.3);
 \draw (b) edge (1.4,-0.1);
 \draw (c) edge (.6,1.2);
 \draw (c) edge (.8,1.16);
 \draw (c) edge (.95,1.05);
 \draw (c) edge (1,.85);
 \draw (d) edge (.6,-1.2);
 \draw (d) edge (.8,-1.16);
 \draw (d) edge (.95,-1.05);
 \draw (d) edge (1,-.85);
\end{tikzpicture},
\end{multline*}
\begin{equation*}
\left(d^0h+hd^0\right)\Gamma  = 2\sum_{x\in V(\Gamma)}h_{v+1}d^0_x\Gamma =
2\sum_{x\in V(\Gamma)}\frac{1}{2}\sum_{e \text{ multiple edge at }x}S(e)\Gamma = 
2\sum_{e \text{ multiple edge}}S(e)\Gamma = 2S(\Gamma)\Gamma.
\end{equation*}
\end{proof}

The lemma ensures that all rows on the first page of our spectral sequences are exact, unless the total strength is $0$, that is we are in the last subcomplex $\GCn_n^\nomul$ of the filtration. Therefore, on the second page, there are all zeros except in the last row where there is $H(\GCn_n^\nomul)$, so the spectral sequences converges to $H(\GCn_n^\nomul)$, what had to be demonstrated.
\end{proof}

\section{Calculating the Euler characteristics}\label{sec:eulerchar}

\subsection{The proof of Theorem \ref{thm:dimgenfun}}
\label{ProofTm1}

Here we calculate the dimension of $V_{v,e,i}^{s_\nu s_\mu s_\rho}$ for $v,e\in\N$, $i\in\{0,1,2,3\}$, $s_\nu,s_\mu\in\{+,*,-\}$ and $s_\rho\in\{+,-\}$. It holds that
\begin{equation}
 \dim\left(V_{v,e,i}^{s_\nu s_\mu s_\rho}\right) 
 = \dim\left(\left(V_{v,e,i}^{s_\nu s_\mu}\right)_{S_v}\right) 
 = \dim\left(\left(V_{v,e,i}^{s_\nu s_\mu}\right)^{S_v}\right)
 = \frac{1}{|S_v|}\sum_{g\in S_v}\xi_{v,e,i}^{s_\nu s_\mu s_\rho}(g),
\end{equation}
where $\xi_{v,e,i}^{s_\nu s_\mu s_\rho}$ is the character of the representation $\rho^{s_\rho}$ of $S_v$ on $V_{v,e,i}^{s_\nu s_\mu}$. Furthermore, it is enough to calculate characters of conjugacy classes of $S_v$, i.e.\ on partitions $\underline{j}$ of $v$:
\begin{equation}
 \dim\left(V_{v,e,i}^{s_\nu s_\mu s_\rho}\right) =
 \frac{1}{v!}\sum_{\substack{j_1,j_2,\dots \geq 0 \\ \sum_\alpha \alpha j_\alpha=v}} \frac{v!}{\prod_\alpha j_\alpha ! \alpha^{j_\alpha}}
 \xi_{v,e,i}^{s_\nu s_\mu s_\rho}(g_{\underline{j}}),
\end{equation}
where $g_{\underline{j}}$ is any element of the conjugacy class $\underline{j}$ of $S_v$.
Let
\begin{equation}\label{equ:8}
P_{i}^{s_\nu s_\mu s_\rho}=\sum_{v,e\geq 0}\dim\left(V_{v,e,i}^{s_\nu s_\mu s_\rho}\right)s^vt^e
\end{equation}
be the generating functions. Then
\begin{equation}
P_{i}^{s_\nu s_\mu s_\rho} = 
\sum_{v,e\geq 0}\sum_{\substack{j_1,j_2,\dots \geq 0 \\ \sum_\alpha \alpha j_\alpha=v}} \frac{1}{\prod_\alpha j_\alpha ! \alpha^{j_\alpha}}
 \xi_{v,e,i}^{s_\nu s_\mu s_\rho}(g_{\underline{j}})s^vt^e =
 \sum_{j_1,j_2,\dots \geq 0} \left( \prod_{\alpha\geq 1}\frac{s^{\alpha j_\alpha}}{ j_\alpha ! \alpha^{j_\alpha}}\right)
 \xi_{\underline{j},i}^{s_\nu s_\mu s_\rho} 
\end{equation}
where
\begin{equation}
\xi_{\underline{j},i}^{s_\nu s_\mu s_\rho} :=
\sum_e \xi_{\sum_\alpha \alpha j_\alpha,e,i}^{s_\nu s_\mu s_\rho}(g_{\underline{j}})t^e
\end{equation}
is the \emph{total character}.

In the following we only consider $(s_\nu, s_\mu, s_\rho)=(+,-,+)$ or $(-,+,-)$. For the computer calculations in subsection \ref{data} we also used another possibilities without writing down the closed formula, whose calculations are similar.

\begin{lemma}
\begin{equation}
\begin{split}
\xi_{\underline{j},0}^{+-+}= & \prod_{\alpha\geq 1} (1+t^-_{2\alpha-1})^{\alpha j_{2\alpha-1}} \\
& \prod_{\alpha\geq 1} \left((1+t^-_\alpha)(1+t^-_{2\alpha})^\alpha\right)^{j_{2\alpha}} \\
& \prod_{\alpha\geq 1} (1+t^-_\alpha)^{\alpha\frac{j_\alpha(j_\alpha-1)}{2}} \\
& \prod_{\beta>\alpha\geq 1}(1+t^-_{\LCM(\alpha,\beta)})^{\GCD(\alpha,\beta)j_\alpha j_\beta},
\end{split}
\end{equation}
where $t^-_\alpha:=-(-t)^\alpha$, and
\begin{equation}
\begin{split}
\xi_{\underline{j},0}^{-+-}= & \prod_{\alpha\geq 1} \left(\frac{1}{1-t^+_{2\alpha-1}}\right)^{(\alpha-1) j_{2\alpha-1}} \\
& \prod_{\alpha\geq 1} \left(-\frac{1}{1+t^+_\alpha}\left(\frac{1}{1-t^+_{2\alpha}}\right)^{\alpha-1}\right)^{j_{2\alpha}} \\
& \prod_{\alpha\geq 1} \left(\frac{1}{1-t^+_\alpha}\right)^{\alpha\frac{j_\alpha(j_\alpha-1)}{2}} \\
& \prod_{\beta>\alpha\geq 1}\left(\frac{1}{1-t^+_{\LCM(\alpha,\beta)}}\right)^{\GCD(\alpha,\beta)j_\alpha j_\beta},
\end{split}
\end{equation}
where $t^+_\alpha:=t^\alpha$.
\end{lemma}
\begin{proof}
The total character $\xi_{\underline{j},0}^{+-+}$ is the polynomial in $t$ with the coefficient of $t^e$ being the character of $g_{\underline{j}}\in S_{\sum_\alpha \alpha j_\alpha}$ in the representation $\rho^+$ on $V_{\sum_\alpha \alpha j_\alpha,e,0}^{+-}$.
A basis of $V_{\sum_\alpha \alpha j_\alpha,e,0}^{+-}$ consists of graphs with $v$ distinguishable vertices and $e$ unoriented indistinguishable edges. Switching edges changes the sign, thus excluding double edges.

An element $g_{\underline{j}}\in S_{\sum_\alpha \alpha j_\alpha}$ acts on $V_{\sum_\alpha \alpha j_\alpha,e,0}^{+-}$ by moving one graph to another. To calculate the character, we need to find graphs $x$ moved to $kx$ for a scalar $k$. By the definition of the representation $k\in\{1,-1\}$. The element $g_{\underline{j}}$ acts in this way on the graph $x$ which has a symmetry (up to the sign) of $g_{\underline{j}}$ on vertices, i.~e.\ whose vertices are partitioned into the $j_\alpha$ cycles of $\alpha$ edges with circular symmetry in every cycle. The cycles are distinguishable, and the beginning vertex in the cycle is marked.

Let us pick a cycle with odd number $2\alpha-1$ of vertices, and let us number the vertices by $1,2,\dots,2\alpha-1$. If there is an edge between vertex $1$ and $l$, because of the symmetry there should be also an edge between vertex $2$ and $l+1$, and so on. We obtain $2\alpha-1$ edges in total. Graphs containing these edges contribute to the total character $\xi_{\underline{j},0}^{+-+}$ by multiplication with $t^-_{2\alpha-1}=-(-t)^{2\alpha-1}=t^{2\alpha-1}$. Note that cycling an odd number of edges is an even permutation, so it does not change sign.
Graphs not containing these edges contribute to the total character by multiplication with $1$, so this possibility contribute by $1+t^-_{2\alpha-1}$ (recall that there are no multiple edges). There are $\alpha$ possibilities of putting that cycle of edges, so the contribution is $(1+t^-_{2\alpha-1})^{\alpha}$. The contributions of all $j_\alpha$ cycles is $(1+t^-_{2\alpha-1})^{\alpha j_{2\alpha-1}}$, and the contribution of all odd cycles is $\prod_{\alpha\geq 1} (1+t^-_{2\alpha-1})^{\alpha j_{2\alpha-1}}$. This is the first line of the formula.

The second line is the contribution of even cycles. The third line is the contribution of the connections between two cycles of the same size, and the forth is the same for cycles of different sizes. The detailed derivation is easy and will be left to the reader. The similar calculation of the total character $\xi_{\underline{j},0}^{-+-}$ will also be left to the reader.
\end{proof}

We are indeed interested in $\xi_{\underline{j},3}^{+-+}$ and $\xi_{\underline{j},3}^{-+-}$. Let us calculate the first and leave the second to the reader. We proceed similarly to the proof of the previous lemma, but after fixing $g_{\underline{j}}\in S_{\sum_\alpha \alpha j_\alpha}$ we should consider only graphs with more than $2$ adjacent edges. Because of the symmetry, the valences of vertices in the same cycle are the same, so we can talk about the valence of the cycle. However, at this point we have not forced the vertices to be at least 3-valent.

What we can do is to construct a graph with some special cycles for which we are sure by the construction that they are at most 2-valent. Let $\overline{\xi}_{\underline{j},3}^{+-+}$ be the \emph{special total character}, i.e.\ `total character' of the partition $\underline{j}$ which allows adding 0, 1 or 2-valent cycles together with edges. So, $\overline{\xi}_{\underline{j},3}^{+-+}$ is the polynomial in variables $t$ and $s^+_\alpha$ for $\alpha\in\N$ where the coefficient next to $t^e\prod_\alpha \left(s^+_\alpha\right)^{n_\alpha}$ is the number of the graphs (counted with appropriate signs) with $j_\alpha$ distinguishable $\alpha$-cycles, $n_\alpha$ indistinguishable special $n_\alpha$-cycles and $e$ edges. All cycles have a marked ``first'' vertex.
If there is a symmetry of the order $r$ between indistinguishable cycles, we divide the term with $r$. 

The key fact is that special cycles with valence up to $2$ can be added to the fixed cycles of partition $\underline{j}$ in a controlled way: they are either disconnected from the rest and form free loops or lines (vacuum), connected to one cycle (antennas) or connect two cycles (connections). This is the reason why we can not calculate the dimension of $V_{v,e,i}^{s_\nu s_\mu s_\rho}$ for $i>3$. Careful calculation leads to the following formula.

\begin{equation}
\begin{split}
\overline{\xi}_{\underline{j},3}^{+-+} = & \xi_{\underline{j},0}^{+-+} \\
& \prod_{\beta>\alpha\geq 1}\prod_{c\geq 1}\prod_{l\geq 1}\exp\left[j_\alpha j_\beta \alpha\beta(c\LCM(\alpha,\beta))^{l-1}\left(s^+_{c\LCM(\alpha,\beta)}\right)^l \left(t^-_{c\LCM(\alpha,\beta)}\right)^{l+1}\right] \qquad
\begin{tikzpicture}[baseline=1.5ex, scale=.5]
\node[cy] (a) at (0,0.4) {};
\node[above left] at (a) {$\scriptstyle \alpha$};
\node[scy] (m1) at (1,1) {};
\node[scy] (m2) at (2,1) {};
\node (m3) at (3,1) {};
\node (ml2) at (4,1) {};
\node[scy] (ml1) at (5,1) {};
\node[cy] (b) at (6,0.4) {};
\node[above right] at (b) {$\scriptstyle \beta$};
\draw (a) edge[bend left=10] (m1);
\draw (m1) edge (m2);
\draw (m2) edge (m3);
\node at (3.5,1) {$\dots$};
\draw (ml2) edge (ml1);
\draw (ml1) edge[bend left=10] (b);
\draw [ultra thin] (1.3,0.7) -- (1.8,0.4);
\draw [ultra thin] (2.2,0.7) -- (2.4,0.4);
\draw [ultra thin] (4.7,0.7) -- (4.2,0.4);
\node at (3,0.1) {$\scriptstyle c\LCM(\alpha,\beta)$};
\draw [decorate,decoration={brace, amplitude=4pt}] (0.7,1.3) -- (5.3,1.3) node[midway, above=3pt]{$\scriptstyle l$};
\end{tikzpicture}\\
& \prod_{\alpha\geq 1}\prod_{c\geq 1}\prod_{l\geq 1}\exp\left[\frac{j_\alpha(j_\alpha-1)}{2}\alpha^2(c\alpha)^{l-1}\left(s^+_{c\alpha}\right)^l \left(t^-_{c\alpha}\right)^{l+1}\right] \qquad\qquad\qquad\qquad
\begin{tikzpicture}[baseline=1.5ex, scale=.5]
\node[cy] (a) at (0,0.4) {};
\node[above left] at (a) {$\scriptstyle \alpha$};
\node[scy] (m1) at (1,1) {};
\node[scy] (m2) at (2,1) {};
\node (m3) at (3,1) {};
\node (ml2) at (4,1) {};
\node[scy] (ml1) at (5,1) {};
\node[cy] (b) at (6,0.4) {};
\node[above right] at (b) {$\scriptstyle \alpha$};
\draw (a) edge[bend left=10] (m1);
\draw (m1) edge (m2);
\draw (m2) edge (m3);
\node at (3.5,1) {$\dots$};
\draw (ml2) edge (ml1);
\draw (ml1) edge[bend left=10] (b);
\node[below=1pt] at (m1) {$\scriptstyle c\alpha$};
\draw [decorate,decoration={brace, amplitude=4pt}] (0.7,1.3) -- (5.3,1.3) node[midway, above=3pt]{$\scriptstyle l$};
\end{tikzpicture}\\
& \prod_{\alpha\geq 1}\prod_{c\geq 1}\prod_{l\geq 2}\exp\left[\frac{1}{2}j_\alpha\alpha^2(c\alpha)^{l-1}\left(s^+_{c\alpha}\right)^l \left(t^-_{c\alpha}\right)^{l+1}\right] \qquad\qquad\qquad\qquad\qquad\qquad
\begin{tikzpicture}[baseline=2ex, scale=.5]
\node[cy] (a) at (3,0) {};
\node[above right] at (a) {$\scriptstyle \alpha$};
\node[scy] (m1) at (1,1) {};
\node[scy] (m2) at (2,1) {};
\node (m3) at (3,1) {};
\node (ml2) at (4,1) {};
\node[scy] (ml1) at (5,1) {};
\draw (a) edge[bend left=20] (m1);
\draw (m1) edge (m2);
\draw (m2) edge (m3);
\node at (3.5,1) {$\dots$};
\draw (ml2) edge (ml1);
\draw (ml1) edge[bend left=20] (a);
\node[below left] at (m1) {$\scriptstyle c\alpha$};
\draw [decorate,decoration={brace, amplitude=4pt}] (0.7,1.3) -- (5.3,1.3) node[midway, above=3pt]{$\scriptstyle l$};
\end{tikzpicture}\\
& \prod_{\alpha\geq 1}\prod_{c\geq 1}\exp\left[j_\alpha\frac{\alpha(\alpha-1)}{2}s^+_{c\alpha} \left(t^-_{c\alpha}\right)^2\right]
\qquad\qquad\qquad\qquad\qquad\qquad\qquad\qquad\qquad
\begin{tikzpicture}[baseline=.5ex, scale=.5]
\node[cy] (a) at (0,0) {};
\node[above left] at (a) {$\scriptstyle \alpha$};
\node[scy] (m1) at (1.5,.5) {};
\draw (a) edge[bend left=20] (m1);
\draw (a) edge[bend right=20] (m1);
\node[below right] at (m1) {$\scriptstyle c\alpha$};
\end{tikzpicture}\\
& \prod_{\alpha\geq 1}\prod_{c\geq 1}\exp\left[j_{2\alpha}\alpha\left(s^+_{(2c-1)\alpha}\right) \left(t^-_{2(2c-1)\alpha}\right)\right]
\qquad\qquad\qquad\qquad\qquad\qquad\qquad\qquad
\begin{tikzpicture}[baseline=.5ex, scale=.5]
\node[cy] (a) at (0,0) {};
\node[above left] at (a) {$\scriptstyle 2\alpha$};
\node[scy] (m1) at (1.5,.5) {};
\draw (a) edge[bend left=10,-left to] (m1);
\node[below right] at (m1) {$\scriptstyle (2c-1)\alpha$};
\end{tikzpicture}\\
& \prod_{\alpha\geq 1}\prod_{c\geq 1}\prod_{l\geq 1}\exp\left[j_\alpha \alpha(c\alpha)^{l-1}\left(s^+_{c\alpha}\right)^l \left(t^-_{c\alpha}\right)^l\right]
\qquad\qquad\qquad\qquad\qquad\qquad\qquad
\begin{tikzpicture}[baseline=2ex, scale=.5]
\node[cy] (a) at (0,.5) {};
\node[above left] at (a) {$\scriptstyle \alpha$};
\node[scy] (m1) at (1,.5) {};
\node[scy] (m2) at (2,.5) {};
\node (m3) at (3,.5) {};
\node (ml2) at (4,.5) {};
\node[scy] (ml1) at (5,.5) {};
\draw (a) edge (m1);
\draw (m1) edge (m2);
\draw (m2) edge (m3);
\node at (3.5,.5) {$\dots$};
\draw (ml2) edge (ml1);
\node[below right] at (ml1) {$\scriptstyle c\alpha$};
\draw [decorate,decoration={brace, amplitude=4pt}] (0.7,.8) -- (5.3,.8) node[midway, above=3pt]{$\scriptstyle l$};
\end{tikzpicture}\\
& \prod_{\alpha\geq 1}\prod_{c\geq 1}\prod_{l\geq 1}\exp\left[j_{2\alpha-1}(2\alpha-1)^l(2c)^{l-1}\left(s^+_{2c(2\alpha-1)}\right)^l\left(t^-_{2c(2\alpha-1)}\right)^l \left(t^-_{c(2\alpha-1)}+c(2\alpha-1)s^+_{c(2\alpha-1)}t^-_{2c(2\alpha-1)}\right)\right] \\
& \qquad\qquad\qquad\qquad\qquad\qquad\qquad\qquad\qquad\qquad\qquad\qquad\quad
\begin{tikzpicture}[baseline=0ex, scale=.5]
\node[cy] (a) at (0,.5) {};
\node[above left] at (a) {$\scriptstyle 2\alpha-1$};
\node[scy] (m1) at (1,1) {};
\node[scy] (m2) at (2,1) {};
\node (m3) at (3,1) {};
\node (ml2) at (4,1) {};
\node[scy] (ml1) at (5,1) {};
\node[scyx] at (ml1) {};
\draw (a) edge[bend left=10] (m1);
\draw (m1) edge (m2);
\draw (m2) edge (m3);
\node at (3.5,1) {$\dots$};
\draw (ml2) edge (ml1);
\node[below right] at (m2) {$\scriptstyle 2c(2\alpha-1)$};
\draw [decorate,decoration={brace, amplitude=4pt}] (0.7,1.3) -- (5.3,1.3) node[midway, above=3pt]{$\scriptstyle l$};
\node at (5,.5) {or};
\node[scy] (me) at (5,0) {};
\node[scy] (me2) at (6,0) {};
\draw (4.25,0) edge (me);
\draw (me) edge[-left to] (me2);
\node[above right] at (me2) {$\scriptstyle c(2\alpha-1)$};
\end{tikzpicture}\\
& \prod_{\alpha\geq 1}\prod_{c\geq 1}\prod_{l\geq 1}\exp\left[j_{2\alpha}(2\alpha)^l(c)^{l-1}\left(s^+_{2c\alpha}\right)^l\left(t^-_{2c\alpha}\right)^l \left(t^-_{c\alpha}+c\alpha s^+_{c\alpha}t^-_{2c\alpha}\right)\right]
\qquad
\begin{tikzpicture}[baseline=1.5ex, scale=.5]
\node[cy] (a) at (0,.5) {};
\node[above left] at (a) {$\scriptstyle 2\alpha$};
\node[scy] (m1) at (1,1) {};
\node[scy] (m2) at (2,1) {};
\node (m3) at (3,1) {};
\node (ml2) at (4,1) {};
\node[scy] (ml1) at (5,1) {};
\node[scyx] at (ml1) {};
\draw (a) edge[bend left=10] (m1);
\draw (m1) edge (m2);
\draw (m2) edge (m3);
\node at (3.5,1) {$\dots$};
\draw (ml2) edge (ml1);
\node[below right] at (m2) {$\scriptstyle 2c\alpha$};
\draw [decorate,decoration={brace, amplitude=4pt}] (0.7,1.3) -- (5.3,1.3) node[midway, above=3pt]{$\scriptstyle l$};
\node at (5,.5) {or};
\node[scy] (me) at (5,0) {};
\node[scy] (me2) at (6,0) {};
\draw (4.25,0) edge (me);
\draw (me) edge[-left to] (me2);
\node[above right] at (me2) {$\scriptstyle c\alpha$};
\end{tikzpicture}\\
& \prod_{c\geq 1}\prod_{l\geq 3}\exp\left[\frac{1}{2l}c^l\left(s^+_c\right)^l \left(t^-_c\right)^l\right]
\qquad\qquad\qquad\qquad\qquad\qquad\qquad\qquad\qquad\qquad
\begin{tikzpicture}[baseline=3ex, scale=.5]
\node[scy] (m1) at (1,1) {};
\node[scy] (m2) at (2,1) {};
\node (m3) at (3,1) {};
\node (ml2) at (4,1) {};
\node[scy] (ml1) at (5,1) {};
\draw (m1) edge[bend right=30] (ml1);
\draw (m1) edge (m2);
\draw (m2) edge (m3);
\node at (3.5,1) {$\dots$};
\draw (ml2) edge (ml1);
\node[below left] at (m1) {$\scriptstyle c$};
\draw [decorate,decoration={brace, amplitude=4pt}] (0.7,1.3) -- (5.3,1.3) node[midway, above=3pt]{$\scriptstyle l$};
\end{tikzpicture}\\
& \prod_{c\geq 1}\prod_{l\geq 2}\exp\left[\frac{1}{2}(2c-1)^{l-1}\left(s^+_{2c-1}\right)^l \left(t^-_{2c-1}\right)^{l-1}\right]
\qquad\qquad\qquad\qquad\qquad\qquad\quad
\begin{tikzpicture}[baseline=3ex, scale=.5]
\node[scy] (m1) at (1,1) {};
\node[scy] (m2) at (2,1) {};
\node (m3) at (3,1) {};
\node (ml2) at (4,1) {};
\node[scy] (ml1) at (5,1) {};
\draw (m1) edge (m2);
\draw (m2) edge (m3);
\node at (3.5,1) {$\dots$};
\draw (ml2) edge (ml1);
\node[below left] at (m1) {$\scriptstyle 2c-1$};
\draw [decorate,decoration={brace, amplitude=4pt}] (0.7,1.3) -- (5.3,1.3) node[midway, above=3pt]{$\scriptstyle l$};
\end{tikzpicture}\\
& \prod_{c\geq 1}\prod_{l\geq 0}\exp\left[\frac{1}{2}(2c)^{l-1}\left(s^+_{2c}\right)^l \left(t^-_{2c}\right)^{l+1}\left(2cs^+_{2c}+2cs^+_{2c}t^-_c+cs^+_c\right)^2\right]
\qquad\qquad\quad
\begin{tikzpicture}[baseline=0ex, scale=.5]
\node[scy] (a1) at (0,1) {};
\node at (0,.5) {or};
\node[scy] (a2) at (0,0) {};
\node[scyx] at (a2) {};
\node at (0,-0.5) {or};
\node[scy] (a3) at (0,-1) {};
\node[scy] (b1) at (6,1) {};
\node at (6,.5) {or};
\node[scy] (b2) at (6,0) {};
\node[scyx] at (b2) {};
\node at (6,-0.5) {or};
\node[scy] (b3) at (6,-1) {};
\node[scy] (m1) at (1,1) {};
\node[scy] (m2) at (2,1) {};
\node (m3) at (3,1) {};
\node (ml2) at (4,1) {};
\node[scy] (ml1) at (5,1) {};
\draw (m1) edge[-left to] (a1);
\draw (a2) edge (0.8,0);
\draw (a3) edge (0.8,-1);
\draw (ml1) edge[-left to] (b1);
\draw (b2) edge (5.2,0);
\draw (b3) edge (5.2,-1);
\draw (m1) edge (m2);
\draw (m2) edge (m3);
\node at (3.5,1) {$\dots$};
\draw (ml2) edge (ml1);
\node[below right] at (m2) {$\scriptstyle 2c$};
\node[above right] at (b1) {$\scriptstyle c$};
\node[above left] at (a1) {$\scriptstyle c$};
\draw [decorate,decoration={brace, amplitude=4pt}] (0.7,1.3) -- (5.3,1.3) node[midway, above=3pt]{$\scriptstyle l$};
\end{tikzpicture}\\
& \prod_{c\geq 1}\exp\left[\frac{1}{2}\frac{c(c-1)}{2}\left(s^+_c\right)^2 \left(t^-_c\right)^2-\frac{c}{4}(s^+_c)^2t^-_{2c}\right]
\qquad\qquad\;\;
\begin{tikzpicture}[baseline=0ex, scale=.5]
\node[scy] (a) at (0,0) {};
\node[above left] at (a) {$\scriptstyle c$};
\node[scy] (b) at (1.5,0) {};
\draw (a) edge[bend left=20] (b);
\draw (a) edge[bend right=20] (b);
\node at (7,0) {and cancelling \begin{tikzpicture}[baseline=-0.5ex]\node[scy] (a) at (0,0) {};\node[scy] (b) at (.5,0) {};\draw (a) edge[left to-left to] (b);\end{tikzpicture} from above};
\end{tikzpicture}\\
& \prod_{c\geq 1}\exp\left[cs^+_{2c} t^-_{2c}\right]
\qquad\qquad\qquad\qquad\qquad\qquad\qquad\qquad\qquad\qquad\qquad\qquad\qquad
\begin{tikzpicture}[baseline=0ex, scale=.5]
\node[scy] (a) at (0,0) {};
\node[scyt] at (a) {};
\node[above left] at (a) {$\scriptstyle 2c$};
\end{tikzpicture}\\
& \prod_{c\geq 1}\exp\left[cs^+_{2c-1} t^-_{2c-1}\right]
\qquad\qquad\qquad\qquad\qquad\qquad\qquad\qquad\qquad\qquad\qquad\qquad
\begin{tikzpicture}[baseline=0ex, scale=.5]
\node[scy] (a) at (0,0) {};
\node[scyt] at (a) {};
\node[above left] at (a) {$\scriptstyle 2c-1$};
\end{tikzpicture}\\
& \prod_{c\geq 1}\exp\left[s^+_{2c} t^-_c\right]
\qquad\qquad\qquad\qquad\qquad\qquad\qquad\qquad\qquad\qquad\qquad\qquad\qquad\quad
\begin{tikzpicture}[baseline=0ex, scale=.5]
\node[scy] (a) at (0,0) {};
\node[scyx] at (a) {};
\node[above left] at (a) {$\scriptstyle 2c$};
\end{tikzpicture}\\
& \prod_{c\geq 1}\exp\left[s^+_c\right]
\qquad\qquad\qquad\qquad\qquad\qquad\qquad\qquad\qquad\qquad\qquad\qquad\qquad\qquad\;
\begin{tikzpicture}[baseline=0ex, scale=.5]
\node[scy] (a) at (0,0) {};
\node[above left] at (a) {$\scriptstyle c$};
\end{tikzpicture}
\end{split}
\end{equation}

The diagrams next to the factors depict the shape from which the factor comes. Full nodes \begin{tikzpicture} \node[cy] at (0,0) {};\end{tikzpicture} represent general cycles in the graph, and empty nodes \begin{tikzpicture} \node[scy] at (0,0) {};\end{tikzpicture} represent special cycles added to the graph, which must be at most 2-valent. Small number next to the node \begin{tikzpicture} \node[cy] (a) at (0,0) {};\node[above left] at (a) {$\scriptstyle \alpha$};\end{tikzpicture} represent the order (number of vertices) of the cycle. A symbol \begin{tikzpicture} \node[scy] (a) at (0,0) {};\node[scyx] at (a) {};\end{tikzpicture} represents a special even cycle with opposite vertices connected, and  \begin{tikzpicture} \node[scy] (a) at (0,0) {};\node[scyt] at (a) {};\end{tikzpicture} represents a special 2-valent cycle with a set of inside edges (not towards the opposite vertex) respecting the symmetry.
A connection \begin{tikzpicture}\node[scy] (a) at (0,0) {};\node[scy] (b) at (.5,0) {};\draw (a) edge[bend left=10,-left to] (b);\end{tikzpicture} (with the right-hand cycle always being the special one) represents a set of edges where every vertex of the right-hand cycle is by symmetry forced to be connected to two opposite vertices of the left-hand even cycle. A connection \begin{tikzpicture}\node[scy] (a) at (0,0) {};\node[scy] (b) at (.5,0) {};\draw (a) edge[bend left=10] (b);\end{tikzpicture} represents a single set of edges connecting vertices from different cycles respecting the symmetry, different from the one represented by a harpoon.

Note that cycles \begin{tikzpicture} \node[scy] (a) at (0,0) {};\node[scyx] at (a) {};\end{tikzpicture} have inner valence $1$, and cycles \begin{tikzpicture} \node[scy] (a) at (0,0) {};\node[scyt] at (a) {};\end{tikzpicture} have inner valence $2$. The existence of a simple connection towards a special cycle \begin{tikzpicture}\node[cy] (a) at (0,0) {};\node[scy] (b) at (.5,0) {};\node[above left] at (a) {$\scriptstyle \alpha$};\node[above right] at (b) {$\scriptstyle \beta$};\draw (a) edge[bend left=10] (b);\end{tikzpicture} implies that $\alpha|\beta$ and increases the valence of the special cycle by one. Therefore a simple connection between two special cycles implies that they have the same order.
A ``harpooned'' connection towards a special cycle \begin{tikzpicture}\node[cy] (a) at (0,0) {};\node[scy] (b) at (.5,0) {};\node[above left] at (a) {$\scriptstyle 2\alpha$};\node[above right] at (b) {$\scriptstyle \beta$};\draw (a) edge[bend left=10,-left to] (b);\end{tikzpicture} implies that $\alpha|\beta$ and increases the valence of the special cycle by two. If the other cycle is also special, its valence is increased by one, and it has a double order.

For the illustration we explain the first factor, the contribution of connections between different cycles, say an $\alpha$-cycle and a $\beta$-cycle, $\beta<\alpha$. The two cycles can be connected via a chain of $l$ special cycles. Because all special cycles are connected to two cycles, there can not be internal connections in the cycles and from each vertex exactly 1 edge goes to the next and to the previous cycle. Because of connecting rules, the order of all special cycles in the chain is the same and it is a multiple of the least common multiple $\LCM(\alpha,\beta)$. So, the contribution of all connections between different cycles is the product over all $\beta>\alpha\geq 1$, all chain lengths $l\geq 1$ and all possibilities of orders of special cycles $c\LCM(\alpha,\beta)$ for $c\geq 1$, of the contribution of such type of connections, i.e.\ of connections between $\alpha$ and $\beta$-cycle of length $l$ and order of special cycle $c\LCM(\alpha,\beta)$.

Let $\xi$ be the contribution of exactly 1 such connection. There can be any number of that type of connections for generally different starting and ending $\alpha$ and $\beta$-cycles. If there are $n$ of them connecting different pairs of cycles, the contribution is $\frac{1}{n!}\xi^n$ in order not to count same cases multiple times. Even if some of them connect the same pair of cycles, because of the symmetry factor the contribution remains $\frac{1}{n!}\xi^n$. So, the total contribution of that type of connection is $1+\xi+\frac{1}{2}\xi^2+\frac{1}{3!}\xi^3+\dots=\exp(\xi)$.

To calculate $\xi$ we first chose an $\alpha$-cycle and $\beta$-cycle in $j_\alpha j_\beta$ possible ways. We can connect the first special cycle with the $\alpha$-cycle in $\alpha$ different ways, and the last one with the $\beta$-cycle in $\beta$ different ways. Connections between special cycles can be done in $(c\LCM(\alpha,\beta))^{l-1}$ different ways. We also add $l$ special cycles $\left(s^+_{c\LCM(\alpha,\beta)}\right)^l$ and $l+1$ cycles of edges $\left(t^-_{c\LCM(\alpha,\beta)}\right)^{l+1}$. Multiplying everything leads to the $\xi$ of the first factor. Other factors are similar.

To get the total character $\xi_{\underline{j},3}^{+-+}$ we start with the total character $\xi_{\underline{j},0}^{+-+}$ of all-valence cycles, and subtract the character of graphs with the same cycles, of which one is special (2- or less-valent). We subtracted graphs with two low-valent cycles twice, so we need to add the character of graphs with 2 special cycles. Than we need to subtract the character of graphs with 3 special cycles, add with 4, etc.

So all special total characters $\overline{\xi}_{\underline{k},3}^{+-+}$ for $\underline{k}\leq\underline{j}$ contribute to the total character $\xi_{\underline{j},3}^{+-+}$, namely the coefficient (a polynomial in $t$) next to $\prod_\alpha \left(s^+_\alpha\right)^{j_\alpha-k_\alpha}$ with a sign $(-1)^{\sum_\alpha j_\alpha-k_\alpha}$. But all cycles in $\xi_{\underline{j},3}^{+-+}$ are distinguishable while the special cycles contributing to $\overline{\xi}_{\underline{k},3}^{+-+}$ are not. We can put $j_\alpha-k_\alpha$ special indistinguishable cycles between $k_\alpha$ ordered cycles in $j_\alpha!/k_\alpha!$ ways. So the contribution of the coefficient next to $\prod_\alpha \left(s^+_\alpha\right)^{j_\alpha-k_\alpha}$ in $\overline{\xi}_{\underline{k},3}^{+-+}$ into $\xi_{\underline{j},3}^{+-+}$ is multiplied by $(-1)^{\sum_\alpha j_\alpha-k_\alpha}k_\alpha!/j_\alpha!$.

Therefore, if we put $s^+_\alpha:=-s^\alpha/\alpha$ we arrive at the formula (see \eqref{equ:8}):
\begin{equation}
P^{even} := P_{3}^{+-+} =
 \sum_{j_1,j_2,\dots \geq 0} \prod_\alpha\frac{s^{\alpha j_\alpha}}{ j_\alpha ! \alpha^{j_\alpha}}
 \xi_{\underline{j},3}^{+-+} =
 \sum_{j_1,j_2,\dots \geq 0} \prod_\alpha\frac{s^{\alpha j_\alpha}}{ j_\alpha ! \alpha^{j_\alpha}}
 \overline{\xi}_{\underline{j},3}^{+-+}.
\end{equation}
We use the $q$-Pochhammer symbol
\begin{equation}
 \QP{a,q} = \prod_{k\geq 0}\left(1-a q^k\right)
\end{equation}
and simplify
\begin{multline}
P^{even}(s,t) = 
\frac{\QP{s, (st)^2}}{\QP{-st,(st)^2}}
\sum_{j_1,j_2,\dots \geq 0}
\prod_\alpha \frac{s^{\alpha j_\alpha}}{j_\alpha! \alpha^{j_\alpha}  }
\frac{1}{\QP{(-st)^\alpha, (-st)^\alpha}^{j_\alpha}}\\
\left( \frac{\QP{(-t)^{2\alpha-1}, (st)^{4\alpha-2}}}{\QP{s^{2\alpha-1}t^{4\alpha-2},(st)^{4\alpha-2}}} \right)^{j_{2\alpha-1}/2}
\left( \frac{\QP{(-t)^{\alpha}, (st)^{2\alpha}}}{ \QP{s^{\alpha}t^{2\alpha},(st)^{2\alpha}}} \right)^{j_{2\alpha}}
\prod_{\alpha, \beta} \QP{(-t)^{\LCM(\alpha, \beta)}, (-st)^{\LCM(\alpha, \beta)}}^{\GCD(\alpha, \beta)j_\alpha j_\beta/2}.
\end{multline}
By an equally tedious and lengthy computation which we leave to the reader one arrives at the formula for the odd case:
\begin{multline}
P^{odd}(s,t) := P_{3}^{-+-} =
\frac{1}{\QP{-s, (st)^2} \QP{(st)^2,(st)^2}}
\sum_{j_1,j_2,\dots \geq 0}
\prod_\alpha \frac{(-s)^{\alpha j_\alpha}}{j_\alpha! (-\alpha)^{j_\alpha}  }
\frac{1}{\QP{(-st)^\alpha, (-st)^\alpha}^{j_\alpha}}\\
\left( \frac{\QP{t^{2\alpha-1}, (st)^{4\alpha-2}}}{\QP{(-s)^{2\alpha-1}t^{4\alpha-2},(st)^{4\alpha-2}}} \right)^{j_{2\alpha-1}/2}
\left( \frac{\QP{t^{\alpha}, (st)^{2\alpha}}}{ \QP{(-s)^{\alpha}t^{2\alpha},(st)^{2\alpha}}} \right)^{j_{2\alpha}}
\prod_{\alpha, \beta} \frac{1}{\QP{t^{\LCM(\alpha, \beta)}, (-st)^{\LCM(\alpha, \beta)}}^{\GCD(\alpha, \beta)j_\alpha j_\beta/2}}.
\end{multline}

\subsection{A variant of Theorem \ref{thm:dimgenfun} }

By a similar computation as that leading to Theorem \ref{thm:dimgenfun} we may also compute generating functions for the dimensions of the spaces of graphs $\VV_{v,e}^{odd*}$ and $\VV_{v,e}^{even*}$.

\begin{thm}\label{thm:dimgenfun2}
Define
\begin{align*}
 P^{odd*}(s,t) &:= \sum_{v,e} \dim\left(\VV_{v,e}^{odd*}\right) s^v t^e &
 P^{even*}(s,t) &:= \sum_{v,e} \dim\left(\VV_{v,e}^{even*}\right) s^v t^e \, .
\end{align*}
Then
\begin{align*}
 P^{odd*}(s,t) =\; & 
\frac{1}{\QP{-s, (st)^2} \QP{(st)^4,(st)^2}}
\sum_{j_1,j_2,\dots \geq 0}
\prod_\alpha \frac{(-s)^{\alpha j_\alpha}}{j_\alpha! (-\alpha)^{j_\alpha}  }
\left(\frac{1-\left(-st^2\right)^\alpha}{\QP{(-st)^\alpha, (-st)^\alpha}}\right)^{j_\alpha}\\
&
\left(\frac{\QP{t^{2\alpha-1}, (st)^{4\alpha-2}}}{\left(1-t^{4\alpha-2}\right)\QP{(-s)^{2\alpha-1}t^{4\alpha-2},(st)^{4\alpha-2}}} \right)^{j_{2\alpha-1}/2}
\left( \frac{\QP{t^{\alpha}, (st)^{2\alpha}}}{\left(1+t^{2\alpha}\right)\QP{(-s)^{\alpha}t^{2\alpha},(st)^{2\alpha}}} \right)^{j_{2\alpha}}\\
&
\prod_{\alpha, \beta} \left(\frac{1-t^{2\LCM(\alpha,\beta)}}{\QP{t^{\LCM(\alpha, \beta)}, (-st)^{\LCM(\alpha, \beta)}}}\right)^{\GCD(\alpha, \beta)j_\alpha j_\beta/2},
\\
 P^{even*}(s,t) =\; & 
\frac{\QP{s, (st)^2}}{\QP{-(st)^3,(st)^2}}
\sum_{j_1,j_2,\dots \geq 0}
\prod_\alpha \frac{s^{\alpha j_\alpha}}{j_\alpha! \alpha^{j_\alpha}  }
\frac{1}{\QP{(-st)^\alpha, (-st)^\alpha}^{j_\alpha}}
\left(\frac{\QP{s^{4\alpha-2}(-t)^{6\alpha-3}, (st)^{4\alpha-2}}}{\left(1+t^{2\alpha-1}\right)\QP{s^{2\alpha-1}t^{4\alpha-2},(st)^{4\alpha-2}}} \right)^{j_{2\alpha-1}/2}\\
&
\left( \frac{\QP{s^{2\alpha}(-t)^{3\alpha}, (st)^{2\alpha}}}{\left(1+(-t)^\alpha\right)\QP{s^{\alpha}t^{2\alpha},(st)^{2\alpha}}} \right)^{j_{2\alpha}}
\prod_{\alpha, \beta} \QP{(-t)^{\LCM(\alpha, \beta)}, (-st)^{\LCM(\alpha, \beta)}}^{\GCD(\alpha, \beta)j_\alpha j_\beta/2}.
\end{align*}
\end{thm}

\subsection{The connected part}
\label{ConnectedP}

Let us denote $n_b^v:=\dim(V_{v,b+v}^{odd})$ and $\tilde n_b^v:=\dim(\tilde V_{v,b+v}^{odd})$. Basis elements of $V_{v,b+v}^{odd}$ are possibly disconnected graphs. Let one of them consist of $j_\alpha^\beta$ connected graphs in $\tilde V_{\beta,\alpha+\beta}^{odd}$, for $\alpha=1,\dots, c$ and $\beta=1,\dots, v$. It is $\sum_{\alpha,\beta}\alpha j_\alpha^\beta=b$ and $\sum_{\alpha,\beta}\beta j_\alpha^\beta=v$. So we are choosing $j_\alpha^\beta$ elements out of $\tilde n_\alpha^\beta$ basis elements of $\tilde V_{\beta,\alpha+\beta}^{odd}$. This can be done in $\binom{\tilde n_\alpha^\beta}{j_\alpha^\beta}$ ways if the number of vertices $\beta$ is odd, and in $(-1)^{j_\alpha^\beta}\binom{-\tilde n_\alpha^\beta}{j_\alpha^\beta}$ ways if $\beta$ is even, respecting the symmetry. All together, we have the formula
\begin{equation}
 n_b^v=\sum_{\substack{\left(j_\alpha^\beta\geq 0|\alpha,\beta=1,2,\dots\right) \\ \sum_{\alpha,\beta}\alpha j_\alpha^\beta=b \\ \sum_{\alpha,\beta}\beta j_\alpha^\beta=v}}
 \prod_{\alpha,\beta}\left(-(-1)^\beta\right)^{j_\alpha^\beta}\binom{-(-1)^\beta\tilde n_\alpha^\beta}{j_\alpha^\beta}.
\end{equation}

One then computes
\begin{equation}
\begin{split}
\chi_b^{odd} & =\sum_{v\geq 0}(-1)^v n_b^v= \\
& = \sum_{v\geq 0}(-1)^v \sum_{\substack{\left(j_\alpha^\beta\geq 0|\alpha,\beta=1,2,\dots\right) \\ \sum_{\alpha,\beta}\alpha j_\alpha^\beta=b \\ \sum_{\alpha,\beta}\beta j_\alpha^\beta=v}}
 \prod_{\alpha,\beta}\left(-(-1)^\beta\right)^{j_\alpha^\beta}\binom{-(-1)^\beta\tilde n_\alpha^\beta}{j_\alpha^\beta}= \\
& = \sum_{\substack{\left(j_\alpha^\beta\geq 0|\alpha,\beta=1,2,\dots\right) \\ \sum_{\alpha,\beta}\alpha j_\alpha^\beta=b}}
 (-1)^{\sum_{\alpha,\beta}\beta j_\alpha^\beta}\prod_{\alpha,\beta}\left(-(-1)^\beta\right)^{j_\alpha^\beta}\binom{-(-1)^\beta\tilde n_\alpha^\beta}{j_\alpha^\beta}= \\
& = \sum_{\substack{\left(j_\alpha^\beta\geq 0|\alpha,\beta=1,2,\dots\right) \\ \sum_{\alpha,\beta}\alpha j_\alpha^\beta=b}}
 \prod_{\alpha,\beta}(-1)^{j_\alpha^\beta}\binom{-(-1)^\beta\tilde n_\alpha^\beta}{j_\alpha^\beta}= \\
& = \sum_{\substack{\left(i_\alpha\geq 0|\alpha=1,2,\dots\right) \\ \sum_{\alpha}\alpha i_\alpha=b}}
 \sum_{\substack{\left(j_\alpha^\beta\geq 0|\alpha,\beta=1,2,\dots\right) \\ \sum_{\beta}j_\alpha^\beta=i_\alpha}}
 \prod_\alpha(-1)^{i_\alpha}\prod_\beta\binom{-(-1)^\beta\tilde n_\alpha^\beta}{j_\alpha^\beta}= \\
 & = \sum_{\underline{i}\in\mathcal{P}(b)}\prod_\alpha(-1)^{i_\alpha}
 \sum_{\substack{\left(j^\beta\geq 0|\beta=1,2,\dots\right) \\ \sum_{\beta}j^\beta=i_\alpha}}
 \prod_\beta\binom{-(-1)^\beta\tilde n_\alpha^\beta}{j_\alpha^\beta}.
\end{split}
\end{equation}

\begin{lemma}
For every $i\in\N$ it holds that
\begin{equation}
 \sum_{\substack{\left(j^\beta\geq 0|\beta=1,2,\dots\right) \\ \sum_\beta j^\beta=i}}\prod_\beta\binom{-(-1)^\beta\tilde n_\alpha^\beta}{j^\beta}=\binom{-\tilde\chi_\alpha^{odd}}{i}.
\end{equation}
\end{lemma}
\begin{proof}
\begin{equation*}
 \sum_ix^i\sum_{\substack{\left(j^\beta\geq 0|\beta=1,2,\dots\right) \\ \sum_\beta j^\beta=i}}\prod_\beta\binom{-(-1)^\beta\tilde n_\alpha^\beta}{j^\beta}=
 \sum_i\sum_{\substack{\left(j^\beta\geq 0|\beta=1,2,\dots\right) \\ \sum_\beta j^\beta=i}}\prod_\beta\binom{-(-1)^\beta\tilde n_\alpha^\beta}{j^\beta}x^{j^\beta}=
 \sum_{\left(j^\beta\geq 0|\beta=1,2,\dots\right)}\prod_\beta\binom{-(-1)^\beta\tilde n_\alpha^\beta}{j^\beta}x^{j^\beta}=
\end{equation*}
\begin{equation*}
 =\prod_\beta\sum_j\binom{-(-1)^\beta\tilde n_\alpha^\beta}{j}x^j=
 \prod_\beta(1+x)^{-(-1)^\beta\tilde n_\alpha^\beta}=
 (1+x)^{-\sum_\beta(-1)^\beta\tilde n_\alpha^\beta}=
 (1+x)^{-\tilde\chi_\alpha^{odd}}=
 \sum_i\binom{-\tilde\chi_\alpha^{odd}}{i}x^i
\end{equation*}
\end{proof}

So the conclusion is that
\begin{equation}
\chi_b^{odd} = \sum_{\underline{i}\in\mathcal{P}(b)}\prod_\alpha(-1)^{i_\alpha}\binom{-\tilde\chi_\alpha^{odd}}{i_\alpha}.
\end{equation}
A similar argument leads to the same formula for the even case:
\begin{equation}
\chi_b^{even} = \sum_{\underline{i}\in\mathcal{P}(b)}\prod_\alpha(-1)^{i_\alpha}\binom{-\tilde\chi_\alpha^{even}}{i_\alpha}.
\end{equation}
The same formulas hold for $\chi_b^{odd*}$ and $\chi_b^{even*}$. They can be used recursively to calculate the Euler characteristics of the complexes of connected graphs from that of the complexes of all graphs.

\subsection{Numerical data}
\label{data}
The formulas from Subsection \ref{ProofTm1} can be used to calculate the dimensions of the spaces $V_{v,e,i}^{s_\nu s_\mu s_\rho}$ and $\tilde V_{v,e,i}^{s_\nu s_\mu s_\rho}$ using the computer. As an example, in Table \ref{tbl:dimensions} we list the dimensions of $\tilde \VV_{v,e}^{even*}$ and $\tilde \VV_{v,e}^{odd*}$ for $v$ up to $24$ and $e$ up to $36$, modulo the product of prime numbers $3999971\cdot 3999949\cdot 3999929\cdot 3999923\cdot 3999917\cdot 3999901\cdot 3999893\cdot 3999971$.
Our results can also be used to calculate the Euler characteristics of the graph complexes $\chi_b^{even}$, $\chi_b^{odd}$, $\chi_b^{even*}$ and $\chi_b^{odd*}$. The formulas from Subsection \ref{ConnectedP} lead us to the Euler characteristics of the connected parts.
We have done these calculations for the even and odd case, with and without tadpoles and multiple edges respectively, for the whole complex and for the connected part, with $b$ up to $30$, modulo $15808115832821291933=991\cdot 997\cdot 3999949\cdot 3999971$. The results are listed in the following table.
Note that the omission of tadpoles or multiple edges does not alter the Euler characteristic. This of course follows for all $b$ from Theorem \ref{thm:intromultiple} and \cite[Proposition 3.4]{grt} and is expected, but we nevertheless provide the computed data below as a consistency check. 

\begin{rem}
 Note in particular that the Euler characteristics of the even and odd graph complexes are astonishingly similar, up to a conventional sign factor.
\end{rem}

\begin{table}[h]
\begin{tabular}{| c | c | c | c | c | c | c | c | c | c |}
\hline
 & \multicolumn{4}{ c| }{Even} & \multicolumn{4}{ c| }{Odd} \\
\cline{2-9}
$b$ & \multicolumn{2}{ c| }{All} & \multicolumn{2}{ c| }{Connected} & \multicolumn{2}{ c| }{All} & \multicolumn{2}{ c| }{Connected}\\
\cline{2-9}
 & $\chi_b^{even}$ & $\chi_b^{even*}$ & $\tilde\chi_b^{even}$ & $\tilde\chi_b^{even*}$ & $\chi_b^{odd}$ & $\chi_b^{odd*}$ & $\tilde\chi_b^{odd}$ & $\tilde\chi_b^{odd*}$ \\
\hline
1 & 0 & 0 & 0 & 0 & 1 & 0 & 1 & 0 \\
\hline
2 & 1 & 1 & 1 & 1 & 2 & 1 & 1 & 1 \\
\hline
3 & 0 & 0 & 0 & 0 & 3 & 1 & 1 & 1 \\
\hline
4 & 2 & 2 & 1 & 1 & 6 & 3 & 2 & 2 \\
\hline
5 & -1 & -1 & -1 & -1 & 8 & 2 & 1 & 1 \\
\hline
6 & 3 & 3 & 1 & 1 & 14 & 6 & 2 & 2 \\
\hline
7 & -1 & -1 & 0 & 0 & 20 & 6 & 2 & 2 \\
\hline
8 & 4 & 4 & 0 & 0 & 32 & 12 & 2 & 2 \\
\hline
9 & -4 & -4 & -2 & -2 & 44 & 12 & 1 & 1 \\
\hline
10 & 6 & 6 & 1 & 1 & 68 & 24 & 3 & 3 \\
\hline
11 & -5 & -5 & 0 & 0 & 93 & 25 & 1 & 1 \\
\hline
12 & 8 & 8 & 0 & 0 & 139 & 46 & 3 & 3 \\
\hline
13 & -10 & -10 & -2 & -2 & 191 & 52 & 4 & 4 \\
\hline
14 & 12 & 12 & 0 & 0 & 274 & 83 & 2 & 2 \\
\hline
15 & -18 & -18 & -4 & -4 & 372 & 98 & 2 & 2 \\
\hline
16 & 12 & 12 & -3 & -3 & 529 & 157 & 6 & 6 \\
\hline
17 & -25 & -25 & -1 & -1 & 713 & 184 & 4 & 4 \\
\hline
18 & 28 & 28 & 8 & 8 & 980 & 267 & -5 & -5 \\
\hline
19 & -25 & -25 & 12 & 12 & 1300 & 320 & -14 & -14 \\
\hline
20 & 62 & 62 & 27 & 27 & 1759 & 459 & -21 & -21 \\
\hline
21 & -22 & -22 & 14 & 14 & 2318 & 559 & -11 & -11 \\
\hline
22 & 56 & 56 & -25 & -25 & 3119 & 801 & 21 & 21 \\
\hline
23 & -74 & -74 & -39 & -39 & 4107 & 988 & 44 & 44 \\
\hline
24 & -396 & -396 & -496 & -496 & 5914 & 1807 & 504 & 504 \\
\hline
25 & -3068 & -3068 & -2979 & -2979 & 10508 & 4594 & 2969 & 2969 \\
\hline
26 & -794 & -794 & -412 & -412 & 13606 & 3098 & 413 & 413 \\
\hline
27 & 35619 & 35619 & 38725 & 38725 & -18948 & -32554 & -38717 & -38717 \\
\hline
28 & 9349 & 9349 & 10583 & 10583 & -21109 & -2161 & -10578 & -10578 \\
\hline
29 & -634587 & -634587 & -667610 & -667610 & 622510 & 643619 & 667596 & 667596 \\
\hline
30 & 39755 & 39755 & 28305 & 28305 & 560813 & -61697 & -28290 & -28290 \\
\hline
\end{tabular}
\caption{\label{tbl:eulerchar} The table of the Euler characteristics of the various graph complexes as defined in \eqref{equ:chidefs}.}
\end{table}






\begin{table}[H]

\noindent\rotatebox{90}{
\scalebox{.472}{
\begin{tabular}{ r | r r r r r r r r r r r r r r r r r r r r r r r r r }
\multicolumn{3}{c}{$Even*$} \\
36 & 0 & 0 & 0 & 0 & 0 & 0 & 0 & 0 & 0 & 0 & 323 & 6914443 & 10238197936 & 2434452003889 & 149691529204800 & 3065797882254941 & 24210784546884942 & 80125688022289219 & 115536095318709248 & 72644240367676642 & 19131543576030878 & 1915629738270074 & 59880793175801 & 377177693363 & 103765564 \\
35 & 0 & 0 & 0 & 0 & 0 & 0 & 0 & 0 & 0 & 0 & 1005 & 12300248 & 11395934706 & 1854245794476 & 81248861159249 & 1206602153008407 & 6924193168254424 & 16471150498499870 & 16647312912300163 & 7021476589491775 & 1152712839432828 & 63299869884371 & 840308953765 & 1094185703 & 0 \\
34 & 0 & 0 & 0 & 0 & 0 & 0 & 0 & 0 & 0 & 0 & 2853 & 20143069 & 11797416222 & 1312699826317 & 40738291624423 & 434093265619034 & 1782503996695968 & 2981432508808535 & 2045883700715935 & 551569517056290 & 52153482139678 & 1353903801649 & 5279400293 & 0 & 0 \\
33 & 0 & 0 & 0 & 0 & 0 & 0 & 0 & 0 & 0 & 0 & 7858 & 30373223 & 11324845175 & 859478584099 & 18738781553491 & 141412445181865 & 407651323249567 & 466336702482262 & 208454338524126 & 33628719414227 & 1628346636468 & 15426982080 & 6247329 & 0 & 0 \\
32 & 0 & 0 & 0 & 0 & 0 & 0 & 0 & 0 & 0 & 0 & 20198 & 42109841 & 10041540904 & 517319664519 & 7841523412407 & 41233491145291 & 81471527376402 & 61513198271050 & 16943477036119 & 1487486705470 & 30453218431 & 59572065 & 0 & 0 & 0 \\
31 & 0 & 0 & 0 & 0 & 0 & 0 & 0 & 0 & 0 & 0 & 46221 & 53552162 & 8183597882 & 284157684349 & 2954983872822 & 10607822460710 & 13933178693051 & 6623831313164 & 1040005051656 & 42886529623 & 257189208 & 0 & 0 & 0 & 0 \\
30 & 0 & 0 & 0 & 0 & 0 & 0 & 0 & 0 & 0 & 3 & 93366 & 62298561 & 6092177954 & 141175882137 & 990244991515 & 2364327746626 & 1983424893300 & 556534619930 & 44349201196 & 664231574 & 412966 & 0 & 0 & 0 & 0 \\
29 & 0 & 0 & 0 & 0 & 0 & 0 & 0 & 0 & 0 & 32 & 168907 & 66042289 & 4111001284 & 62747495599 & 290471343715 & 445893643989 & 226333613711 & 34141128298 & 1141986095 & 3521433 & 0 & 0 & 0 & 0 & 0 \\
28 & 0 & 0 & 0 & 0 & 0 & 0 & 0 & 0 & 1 & 127 & 276458 & 63413970 & 2490799385 & 24610685274 & 73085889211 & 68938018460 & 19620090166 & 1375793368 & 13412419 & 0 & 0 & 0 & 0 & 0 & 0 \\
27 & 0 & 0 & 0 & 0 & 0 & 0 & 0 & 0 & 1 & 311 & 407915 & 54684111 & 1339061158 & 8371041155 & 15357400905 & 8360542063 & 1190269491 & 30068745 & 30528 & 0 & 0 & 0 & 0 & 0 & 0 \\
26 & 0 & 0 & 0 & 0 & 0 & 0 & 0 & 0 & 0 & 684 & 537528 & 41907083 & 629274575 & 2413785991 & 2598217463 & 745263236 & 43995862 & 227577 & 0 & 0 & 0 & 0 & 0 & 0 & 0 \\
25 & 0 & 0 & 0 & 0 & 0 & 0 & 0 & 0 & 0 & 1545 & 627110 & 28187228 & 253575578 & 572310253 & 335860329 & 44006088 & 749459 & 0 & 0 & 0 & 0 & 0 & 0 & 0 & 0 \\
24 & 0 & 0 & 0 & 0 & 0 & 0 & 0 & 0 & 1 & 3213 & 641553 & 16379726 & 85417799 & 106923865 & 30618203 & 1424662 & 2526 & 0 & 0 & 0 & 0 & 0 & 0 & 0 & 0 \\
23 & 0 & 0 & 0 & 0 & 0 & 0 & 0 & 0 & 7 & 5625 & 567604 & 8048887 & 23223881 & 14779073 & 1718769 & 16244 & 0 & 0 & 0 & 0 & 0 & 0 & 0 & 0 & 0 \\
22 & 0 & 0 & 0 & 0 & 0 & 0 & 0 & 0 & 17 & 8208 & 426113 & 3246800 & 4844186 & 1365037 & 45030 & 0 & 0 & 0 & 0 & 0 & 0 & 0 & 0 & 0 & 0 \\
21 & 0 & 0 & 0 & 0 & 0 & 0 & 0 & 0 & 28 & 10175 & 265083 & 1032036 & 716981 & 69847 & 221 & 0 & 0 & 0 & 0 & 0 & 0 & 0 & 0 & 0 & 0 \\
20 & 0 & 0 & 0 & 0 & 0 & 0 & 0 & 0 & 65 & 10658 & 132526 & 243759 & 66081 & 1274 & 0 & 0 & 0 & 0 & 0 & 0 & 0 & 0 & 0 & 0 & 0 \\
19 & 0 & 0 & 0 & 0 & 0 & 0 & 0 & 0 & 155 & 9033 & 50863 & 38916 & 2938 & 0 & 0 & 0 & 0 & 0 & 0 & 0 & 0 & 0 & 0 & 0 & 0 \\
18 & 0 & 0 & 0 & 0 & 0 & 0 & 0 & 0 & 252 & 5849 & 13867 & 3497 & 30 & 0 & 0 & 0 & 0 & 0 & 0 & 0 & 0 & 0 & 0 & 0 & 0 \\
17 & 0 & 0 & 0 & 0 & 0 & 0 & 0 & 0 & 291 & 2742 & 2329 & 110 & 0 & 0 & 0 & 0 & 0 & 0 & 0 & 0 & 0 & 0 & 0 & 0 & 0 \\
16 & 0 & 0 & 0 & 0 & 0 & 0 & 0 & 4 & 262 & 879 & 188 & 0 & 0 & 0 & 0 & 0 & 0 & 0 & 0 & 0 & 0 & 0 & 0 & 0 & 0 \\
15 & 0 & 0 & 0 & 0 & 0 & 0 & 1 & 14 & 179 & 170 & 7 & 0 & 0 & 0 & 0 & 0 & 0 & 0 & 0 & 0 & 0 & 0 & 0 & 0 & 0 \\
14 & 0 & 0 & 0 & 0 & 0 & 0 & 1 & 16 & 75 & 13 & 0 & 0 & 0 & 0 & 0 & 0 & 0 & 0 & 0 & 0 & 0 & 0 & 0 & 0 & 0 \\
13 & 0 & 0 & 0 & 0 & 0 & 0 & 0 & 10 & 12 & 0 & 0 & 0 & 0 & 0 & 0 & 0 & 0 & 0 & 0 & 0 & 0 & 0 & 0 & 0 & 0 \\
12 & 0 & 0 & 0 & 0 & 0 & 0 & 0 & 6 & 1 & 0 & 0 & 0 & 0 & 0 & 0 & 0 & 0 & 0 & 0 & 0 & 0 & 0 & 0 & 0 & 0 \\
11 & 0 & 0 & 0 & 0 & 0 & 0 & 1 & 1 & 0 & 0 & 0 & 0 & 0 & 0 & 0 & 0 & 0 & 0 & 0 & 0 & 0 & 0 & 0 & 0 & 0 \\
10 & 0 & 0 & 0 & 0 & 0 & 0 & 2 & 0 & 0 & 0 & 0 & 0 & 0 & 0 & 0 & 0 & 0 & 0 & 0 & 0 & 0 & 0 & 0 & 0 & 0 \\
9 & 0 & 0 & 0 & 0 & 0 & 0 & 0 & 0 & 0 & 0 & 0 & 0 & 0 & 0 & 0 & 0 & 0 & 0 & 0 & 0 & 0 & 0 & 0 & 0 & 0 \\
8 & 0 & 0 & 0 & 0 & 0 & 0 & 0 & 0 & 0 & 0 & 0 & 0 & 0 & 0 & 0 & 0 & 0 & 0 & 0 & 0 & 0 & 0 & 0 & 0 & 0 \\
7 & 0 & 0 & 0 & 0 & 0 & 0 & 0 & 0 & 0 & 0 & 0 & 0 & 0 & 0 & 0 & 0 & 0 & 0 & 0 & 0 & 0 & 0 & 0 & 0 & 0 \\
6 & 0 & 0 & 0 & 0 & 1 & 0 & 0 & 0 & 0 & 0 & 0 & 0 & 0 & 0 & 0 & 0 & 0 & 0 & 0 & 0 & 0 & 0 & 0 & 0 & 0 \\
5 & 0 & 0 & 0 & 0 & 0 & 0 & 0 & 0 & 0 & 0 & 0 & 0 & 0 & 0 & 0 & 0 & 0 & 0 & 0 & 0 & 0 & 0 & 0 & 0 & 0 \\
4 & 0 & 0 & 0 & 0 & 0 & 0 & 0 & 0 & 0 & 0 & 0 & 0 & 0 & 0 & 0 & 0 & 0 & 0 & 0 & 0 & 0 & 0 & 0 & 0 & 0 \\
3 & 0 & 0 & 0 & 0 & 0 & 0 & 0 & 0 & 0 & 0 & 0 & 0 & 0 & 0 & 0 & 0 & 0 & 0 & 0 & 0 & 0 & 0 & 0 & 0 & 0 \\
2 & 0 & 0 & 0 & 0 & 0 & 0 & 0 & 0 & 0 & 0 & 0 & 0 & 0 & 0 & 0 & 0 & 0 & 0 & 0 & 0 & 0 & 0 & 0 & 0 & 0 \\
1 & 0 & 0 & 0 & 0 & 0 & 0 & 0 & 0 & 0 & 0 & 0 & 0 & 0 & 0 & 0 & 0 & 0 & 0 & 0 & 0 & 0 & 0 & 0 & 0 & 0 \\
0 & 1 & 0 & 0 & 0 & 0 & 0 & 0 & 0 & 0 & 0 & 0 & 0 & 0 & 0 & 0 & 0 & 0 & 0 & 0 & 0 & 0 & 0 & 0 & 0 & 0 \\
\hline
  & 0 & 1 & 2 & 3 & 4 & 5 & 6 & 7 & 8 & 9 & 10 & 11 & 12 & 13 & 14 & 15 & 16 & 17 & 18 & 19 & 20 & 21 & 22 & 23 & 24 \\
\multicolumn{2}{}{} \\
\end{tabular}}}
\noindent\rotatebox{90}{
\scalebox{.472}{
\begin{tabular}{ r | r r r r r r r r r r r r r r r r r r r r r r r r r }
\multicolumn{3}{c}{$Odd*$} \\
36 & 0 & 0 & 0 & 0 & 0 & 0 & 0 & 0 & 0 & 1 & 478 & 7184889 & 10264841321 & 2428566238061 & 149215526673547 & 3058791031411277 & 24190059107290419 & 80161353715143616 & 115688062454536997 & 72766334437488265 & 19162018671075010 & 1917876403039649 & 59916327855165 & 377215786591 & 104035119 \\
35 & 0 & 0 & 0 & 0 & 0 & 0 & 0 & 0 & 0 & 0 & 1476 & 12656450 & 11409106391 & 1848695699280 & 80990011921765 & 1204350598200103 & 6922376949016458 & 16487604325642019 & 16675515572742650 & 7034270197508544 & 1154435830441245 & 63356370706979 & 840549193304 & 1094411650 & 0 \\
34 & 0 & 0 & 0 & 0 & 0 & 0 & 0 & 0 & 0 & 1 & 3782 & 20544991 & 11795953743 & 1308152524469 & 40616202657361 & 433537414482164 & 1783236020595934 & 2986023303866542 & 2049945027230478 & 552578062244414 & 52218942157741 & 1354622849653 & 5278660177 & 0 & 0 \\
33 & 0 & 0 & 0 & 0 & 0 & 0 & 0 & 0 & 0 & 1 & 9999 & 30816037 & 11308785524 & 856238088390 & 18690570825206 & 141336375950228 & 408109763954217 & 467278000530782 & 208900734518827 & 33684818059654 & 1629759318748 & 15428831783 & 6302088 & 0 & 0 \\
32 & 0 & 0 & 0 & 0 & 0 & 0 & 0 & 0 & 0 & 5 & 22946 & 42496925 & 10015495299 & 515352276433 & 7826655271164 & 41247052313556 & 81620319199596 & 61659409155502 & 16979226824941 & 1489440289747 & 30467187002 & 59607306 & 0 & 0 & 0 \\
31 & 0 & 0 & 0 & 0 & 0 & 0 & 0 & 0 & 0 & 6 & 50341 & 53888311 & 8153609941 & 283162797859 & 2952125602278 & 10621286120700 & 13967209256759 & 6640671981615 & 1041935068084 & 42921726596 & 257014269 & 0 & 0 & 0 & 0 \\
30 & 0 & 0 & 0 & 0 & 0 & 0 & 0 & 0 & 0 & 26 & 98122 & 62477351 & 6065877636 & 140785143630 & 990407963724 & 2369520890328 & 1989161366018 & 557901090988 & 44407460914 & 664289471 & 424465 & 0 & 0 & 0 & 0 \\
29 & 0 & 0 & 0 & 0 & 0 & 0 & 0 & 0 & 0 & 49 & 176937 & 66063551 & 4092424661 & 62644345618 & 290887773316 & 447239696071 & 227024876844 & 34209898447 & 1142991182 & 3528730 & 0 & 0 & 0 & 0 & 0 \\
28 & 0 & 0 & 0 & 0 & 0 & 0 & 0 & 0 & 1 & 165 & 284646 & 63216765 & 2481023713 & 24607633969 & 73281944226 & 69183841838 & 19676613873 & 1377692570 & 13378933 & 0 & 0 & 0 & 0 & 0 & 0 \\
27 & 0 & 0 & 0 & 0 & 0 & 0 & 0 & 0 & 0 & 345 & 416756 & 54431646 & 1335414915 & 8384065789 & 15415999361 & 8391997470 & 1192670365 & 30060258 & 32701 & 0 & 0 & 0 & 0 & 0 & 0 \\
26 & 0 & 0 & 0 & 0 & 0 & 0 & 0 & 0 & 1 & 899 & 541142 & 41679653 & 628888950 & 2422021662 & 2609841863 & 747657371 & 44083731 & 228906 & 0 & 0 & 0 & 0 & 0 & 0 & 0 \\
25 & 0 & 0 & 0 & 0 & 0 & 0 & 0 & 0 & 1 & 1768 & 629051 & 28088250 & 253989803 & 575085565 & 337492002 & 44125758 & 744159 & 0 & 0 & 0 & 0 & 0 & 0 & 0 & 0 \\
24 & 0 & 0 & 0 & 0 & 0 & 0 & 0 & 0 & 5 & 3539 & 638053 & 16346366 & 85804143 & 107579823 & 30713501 & 1422641 & 3056 & 0 & 0 & 0 & 0 & 0 & 0 & 0 & 0 \\
23 & 0 & 0 & 0 & 0 & 0 & 0 & 0 & 0 & 5 & 5655 & 565597 & 8065049 & 23361947 & 14854653 & 1727020 & 16468 & 0 & 0 & 0 & 0 & 0 & 0 & 0 & 0 & 0 \\
22 & 0 & 0 & 0 & 0 & 0 & 0 & 0 & 0 & 24 & 8455 & 423591 & 3261195 & 4888420 & 1373151 & 43891 & 0 & 0 & 0 & 0 & 0 & 0 & 0 & 0 & 0 & 0 \\
21 & 0 & 0 & 0 & 0 & 0 & 0 & 0 & 1 & 36 & 10175 & 266543 & 1043984 & 720525 & 69551 & 364 & 0 & 0 & 0 & 0 & 0 & 0 & 0 & 0 & 0 & 0 \\
20 & 0 & 0 & 0 & 0 & 0 & 0 & 0 & 0 & 100 & 10780 & 132923 & 245866 & 67271 & 1331 & 0 & 0 & 0 & 0 & 0 & 0 & 0 & 0 & 0 & 0 & 0 \\
19 & 0 & 0 & 0 & 0 & 0 & 0 & 0 & 1 & 146 & 8828 & 51862 & 39501 & 2669 & 0 & 0 & 0 & 0 & 0 & 0 & 0 & 0 & 0 & 0 & 0 & 0 \\
18 & 0 & 0 & 0 & 0 & 0 & 0 & 0 & 1 & 266 & 5961 & 13798 & 3406 & 64 & 0 & 0 & 0 & 0 & 0 & 0 & 0 & 0 & 0 & 0 & 0 & 0 \\
17 & 0 & 0 & 0 & 0 & 0 & 0 & 0 & 4 & 273 & 2773 & 2519 & 126 & 0 & 0 & 0 & 0 & 0 & 0 & 0 & 0 & 0 & 0 & 0 & 0 & 0 \\
16 & 0 & 0 & 0 & 0 & 0 & 0 & 0 & 4 & 303 & 955 & 153 & 0 & 0 & 0 & 0 & 0 & 0 & 0 & 0 & 0 & 0 & 0 & 0 & 0 & 0 \\
15 & 0 & 0 & 0 & 0 & 0 & 0 & 1 & 14 & 156 & 154 & 13 & 0 & 0 & 0 & 0 & 0 & 0 & 0 & 0 & 0 & 0 & 0 & 0 & 0 & 0 \\
14 & 0 & 0 & 0 & 0 & 0 & 0 & 0 & 11 & 84 & 13 & 0 & 0 & 0 & 0 & 0 & 0 & 0 & 0 & 0 & 0 & 0 & 0 & 0 & 0 & 0 \\
13 & 0 & 0 & 0 & 0 & 0 & 0 & 1 & 15 & 5 & 0 & 0 & 0 & 0 & 0 & 0 & 0 & 0 & 0 & 0 & 0 & 0 & 0 & 0 & 0 & 0 \\
12 & 0 & 0 & 0 & 0 & 0 & 0 & 0 & 5 & 5 & 0 & 0 & 0 & 0 & 0 & 0 & 0 & 0 & 0 & 0 & 0 & 0 & 0 & 0 & 0 & 0 \\
11 & 0 & 0 & 0 & 0 & 0 & 0 & 3 & 2 & 0 & 0 & 0 & 0 & 0 & 0 & 0 & 0 & 0 & 0 & 0 & 0 & 0 & 0 & 0 & 0 & 0 \\
10 & 0 & 0 & 0 & 0 & 0 & 1 & 0 & 0 & 0 & 0 & 0 & 0 & 0 & 0 & 0 & 0 & 0 & 0 & 0 & 0 & 0 & 0 & 0 & 0 & 0 \\
9 & 0 & 0 & 0 & 0 & 0 & 0 & 1 & 0 & 0 & 0 & 0 & 0 & 0 & 0 & 0 & 0 & 0 & 0 & 0 & 0 & 0 & 0 & 0 & 0 & 0 \\
8 & 0 & 0 & 0 & 0 & 0 & 0 & 0 & 0 & 0 & 0 & 0 & 0 & 0 & 0 & 0 & 0 & 0 & 0 & 0 & 0 & 0 & 0 & 0 & 0 & 0 \\
7 & 0 & 0 & 0 & 0 & 0 & 0 & 0 & 0 & 0 & 0 & 0 & 0 & 0 & 0 & 0 & 0 & 0 & 0 & 0 & 0 & 0 & 0 & 0 & 0 & 0 \\
6 & 0 & 0 & 0 & 0 & 1 & 0 & 0 & 0 & 0 & 0 & 0 & 0 & 0 & 0 & 0 & 0 & 0 & 0 & 0 & 0 & 0 & 0 & 0 & 0 & 0 \\
5 & 0 & 0 & 0 & 0 & 0 & 0 & 0 & 0 & 0 & 0 & 0 & 0 & 0 & 0 & 0 & 0 & 0 & 0 & 0 & 0 & 0 & 0 & 0 & 0 & 0 \\
4 & 0 & 0 & 0 & 0 & 0 & 0 & 0 & 0 & 0 & 0 & 0 & 0 & 0 & 0 & 0 & 0 & 0 & 0 & 0 & 0 & 0 & 0 & 0 & 0 & 0 \\
3 & 0 & 0 & 0 & 0 & 0 & 0 & 0 & 0 & 0 & 0 & 0 & 0 & 0 & 0 & 0 & 0 & 0 & 0 & 0 & 0 & 0 & 0 & 0 & 0 & 0 \\
2 & 0 & 0 & 0 & 0 & 0 & 0 & 0 & 0 & 0 & 0 & 0 & 0 & 0 & 0 & 0 & 0 & 0 & 0 & 0 & 0 & 0 & 0 & 0 & 0 & 0 \\
1 & 0 & 0 & 0 & 0 & 0 & 0 & 0 & 0 & 0 & 0 & 0 & 0 & 0 & 0 & 0 & 0 & 0 & 0 & 0 & 0 & 0 & 0 & 0 & 0 & 0 \\
0 & 1 & 0 & 0 & 0 & 0 & 0 & 0 & 0 & 0 & 0 & 0 & 0 & 0 & 0 & 0 & 0 & 0 & 0 & 0 & 0 & 0 & 0 & 0 & 0 & 0 \\
\hline
  & 0 & 1 & 2 & 3 & 4 & 5 & 6 & 7 & 8 & 9 & 10 & 11 & 12 & 13 & 14 & 15 & 16 & 17 & 18 & 19 & 20 & 21 & 22 & 23 & 24 \\
\end{tabular}
}}

\caption{\label{tbl:dimensions} Dimensions of the spaces of connected graphs $\tilde \VV_{v,e}^{even*}$ and $\tilde \VV_{v,e}^{odd*}$ as computed by Theorem \ref{thm:dimgenfun2} and (a version of) the formulas of section \ref{ConnectedP}. }
\end{table}

\bibliography{biblio}{}
\bibliographystyle{plain}

\end{document}